\documentclass[reqno,11pt,oneside]{amsart}
\usepackage{amsmath, amssymb, amsthm, amsfonts} 
\usepackage[utf8]{inputenc}
\usepackage[english]{babel}
\usepackage[T1]{fontenc}
\usepackage{amsmath,amsfonts,amssymb,latexsym,amsthm,dsfont,amsbsy,stmaryrd}
\usepackage{tikz}
\usepackage{mathrsfs}

\usepackage{xcolor,graphicx}

\usepackage{enumitem}

\usepackage{mathtools}
\mathtoolsset{showonlyrefs}

\definecolor{myblue}{RGB}{0,51,157}
\usepackage{hyperref}
\hypersetup{colorlinks=true,linkcolor=myblue,citecolor=blue!50,urlcolor=green} 
\usepackage{url}

\usepackage{fancyhdr}

\usepackage{geometry}

\renewcommand{\theequation}{\arabic{section}.\arabic{equation}}
\let \ssection=\section
\renewcommand{\section}{\setcounter{equation}{0}\ssection}

\allowdisplaybreaks
\renewcommand{\leq}{\leqslant}
\renewcommand{\geq}{\geqslant}

\newtheorem{thm}{Theorem}[section]
\newtheorem{corollary}[thm]{Corollary}
\newtheorem{proposition}[thm]{Proposition}

\newtheorem{lemma}[thm]{Lemma}

\newtheorem{exple}[thm]{Example}
\newtheorem{remark}[thm]{Remark}
\newtheorem{assumption}{Assumption}

\newcommand{\ind}{\mathds{1}}
\newcommand{\var}{\mathrm{Var}}
\newcommand{\cov}{\mathrm{Cov}}

\newcommand{\dE}{\mathbb{E}}

\newcommand{\dP}{\mathbb{P}}
\newcommand{\dR}{\mathbb{R}}

\newcommand{\dZZ}{\mathscr{Z}}

\newcommand{\cF}{\mathcal{F}}
\newcommand{\cG}{\mathcal{G}}

\newcommand{\bM}{\mathscr M}

\newcommand{\rA}{{\mathrm A}}
\newcommand{\rB}{{\mathrm B}}
\newcommand{\rc}{{\mathrm c}}
\newcommand{\rd}{{\mathrm d}}
\newcommand{\re}{{\mathrm e}}
\newcommand{\ri}{{\mathrm i}}

\newcommand{\ABS}[1]{{{\left| #1 \right|}}} 
\newcommand{\BRA}[1]{{{\left\{#1\right\}}}} 
\newcommand{\PAR}[1]{{{\left(#1\right)}}} 
\newcommand{\SBRA}[1]{{{\left[#1\right]}}} 

\title[Insurance Risk Models in a heterogeneous time-dependent population]{Insurance risk models in a heterogeneous time-dependent population: scaling limits and ruin probabilities}
 
\author{Hélène Guérin}
\address{D\'epartement de math\'ematiques, Universit\'e du Qu\'ebec \`a Montr\'eal (UQAM), 201 av.\ Pr\'esident-Kennedy, Montr\'eal (Qu\'ebec) H2X 3Y7, Canada}
\curraddr{}
\email{guerin.helene@uqam.ca, renaud.jf@uqam.ca, bricezotsa@gmail.com}
\thanks{}

\author{Michel Mandjes}
\address{Mathematical Institute, Leiden University, Einsteinweg 55, 2333 CC Leiden, the Netherlands}
\curraddr{}
\email{m.r.h.mandjes@math.leidenuniv.nl}
\thanks{}

\author{Jean-François Renaud}

\author{Arsene Brice Zotsa Ngoufack}

\thanks{Funding in support of this work was provided by  Discovery Grants (RGPIN-2025-05758, RGPIN-2020-07239) from the Natural Sciences and Engineering Research Council of Canada (NSERC). ABZN acknowledges funding from the SCOR Foundation for Science}
\thanks{This work was supported by the Research Institute for Mathematical Sciences,
an International Joint Usage/Research Center located in Kyoto University.}
\date{\today}

\begin{document}

\maketitle

\begin{abstract}
Epidemic dynamics introduce time-varying heterogeneity into insured populations, as individuals’ risk profiles depend on their evolving health status, thereby challenging classical insurance models based on homogeneity. Motivated by this phenomenon, we develop in this paper actuarial risk models in which the claim process is directly linked to an underlying general stochastic population dynamics, resulting in interacting subpopulations with different claim frequencies and severities. We propose both collective and individual modeling frameworks that track the changing composition of the insured population and its impact on aggregate claims. For these models, we derive scaling limits and provide bounds and approximations for ruin probabilities, offering tractable tools for risk assessment. Applications in an SIS context are detailed for different classes of risk processes, thereby contributing to a better understanding of how contagion-driven changes in population structure affect insurer solvency and supporting more realistic modeling of insurance portfolios in the presence of epidemic risk.
\end{abstract}

\medskip 

\noindent \textit{MSC2020 subject classifications:} 91B05, 91G05, 60G51, 60K37, 92D30 

\medskip


\section{Introduction}

Epidemics introduce a structured and dynamic form of heterogeneity into insured populations, as individuals’ risk profiles depend on their health status which evolves over time. In particular, the coexistence of infected and non-infected subpopulations gives rise to substantially different claim arrival intensities, claim size distributions, and temporal dependence structures. This is a challenge for classical insurance risk models, which typically rely on homogeneity assumptions and independence across insured individuals.

Motivated by recent pandemics and the increasing interaction between epidemiological dynamics and insurance outcomes, this paper investigates insurance risk models in which the claim process is coupled with an underlying population dynamics. While epidemic models provide a primary motivation and illustrate our main examples, the mathematical framework developed here applies more generally to population-based dynamics with state-dependent claim mechanisms. By embedding such population dynamics into actuarial risk frameworks, we aim to capture how transitions between population states affect the aggregate claims process and, ultimately, insurer solvency. Our analysis focuses on two key aspects: scaling limits for the resulting risk processes and the derivation of bounds and asymptotic approximations for the corresponding ruin probabilities.

\medskip
 
{\it Multiclass population.} In our setup, an insurance portfolio consisting of a population is partitioned into sub-populations. 
For analytical tractability, we restrict our attention to two sub-populations, whose associated reserve processes are consistently indexed by $\rA$ and $\rB$; this restriction is imposed solely for simplicity, as the framework extends naturally to any finite number of sub-populations. We assume each subpopulation is homogeneous in terms of its risk characteristics. As mentioned above, the two sub-populations may consist of \textit{infectious} and \textit{susceptible} individuals in an epidemiological setting, but they may also represent risk-seeking and risk-averse individuals, employed and unemployed individuals, insureds with or without a particular disease, or individuals with or without a healthy lifestyle, among many other possibilities.

A key element in our framework is the modeling of the (stochastically) evolving number of individuals (or risks) in each sub-population. For that purpose, let the stochastic process $F=(F(t))_{t\geq 0}$ describe the fraction of risks of type $\rA$. Clearly, the random variable $1-F(t)$ is then the fraction of risks of type $\rB$ at time $t$. If the size of the whole population is known and equal to an integer $N$, then we can consider a stochastic process $I=(I(t))_{t \geq 0}$ where $I(t)$ is the number of individuals of type $\rA$ at time $t$. In this case, we have $F(t)={I(t)}/{N}$. The dynamics of these processes will be specified in greater detail below.

In this paper, we define two types of models: a {\it collective model} based on the fraction of risks $F$ that are of class $\rA$, and an {\it individual model} based on the fraction of individuals $I/N$ that are in class $\rA$.

\medskip

{\it Stochastic objects.}
We proceed by introducing the model primitives through a collection of stochastic objects; this framework will be maintained throughout the article.
\begin{itemize}
\item[$\circ$]
We fix a probability space $\PAR{\Omega, \cF, \dP}$ on which are defined two Lévy processes $X=\PAR{X(t)}_{t\geq 0}$ and $Y=\PAR{Y(t)}_{t\geq 0}$, associated to class $\rA$ and class $\rB$ respectively, and a stochastic process $F=(F(t))_{t\geq 0}$ taking values in $[0,1]$ and having \textit{càdlàg} sample paths, describing the fraction of risks of type $\rA$. 
Let $\cF=\PAR{\cF_t}_{t \geq 0}$ be the filtration generated by those three processes.
\item[$\circ$] We assume $X,Y,F$ are independent. 
\item[$\circ$] For a finite population of size $N\in\mathbb{N}$, define $I=\PAR{I(t)}_{t \geq 0}$ by $I(t):=NF(t)$. Clearly, this process is taking values in $\{0,1,\ldots,N\}$.
\end{itemize}

Note that, apart from being independent of the risk processes $X$ and $Y$ and having càdlàg trajectories, no additional assumptions are imposed on the process $F$ (and thus on $I$). In particular, we do not assume that $F$ is Markovian.

Henceforth, let $\varphi_X$ and $\varphi_Y$ be the characteristic exponents of $X$ and $Y$, i.e.,
    \[
    \dE\SBRA{\re^{i\alpha X(t)}}=\re^{\varphi_X(\alpha)t}\quad \text{and} \quad \dE\SBRA{\re^{i\alpha Y(t)}}=\re^{\varphi_Y(\alpha)t}.
    \]
 It is well known, by the {\it L\'evy-Khinchine formula} \cite[Theorem 1.3]{Kyprianou2014}, that these characteristic exponents are of the form: for $\alpha \in \mathbb{R}$,
\[
\varphi_X(\alpha) = \rc_X i\alpha - \frac{\sigma_X^2}{2} \alpha^2 - \int_{\dR} \left(1 - e^{i \alpha x} + i \alpha x \ind_{\ABS{x}<1} \right) \nu_X(\rd x) ,
\]
with $(\rc_X, \sigma_X, \nu_X)$ the {\it L\'evy triple} associated with $\varphi_X$. 
Here, $\rc_X \in \dR$ and $\sigma_X \geq 0$, and $\nu_X$ is the Lévy measure of $X$, i.e., a sigma-finite measure on $\dR \setminus \{0\}$ such that $\int_{\dR} \PAR{1 \wedge x^2} \nu_X(\rd x) < \infty$. Note that if $X$ has a finite second moment, then $\varphi_X'(0)=i\,\dE[X(1)]$ and $\varphi_X''(0)=-\var(X(1))$. Similarly, the Lévy triple associated with $\varphi_Y$ is given by $(\rc_Y, \sigma_Y, \nu_Y)$.

We finish our account of the model primitives by a notational comment. Throughout this paper, for a process $Z=\PAR{Z(t)}_{t\geq 0}$, we denote by $\cF^Z$ the $\sigma$-algebra generated by $\PAR{Z(t)}_{t\geq 0}$. Likewise, for a given $t \in [0,\infty)$, we denote by $\cF^Z_t$ the $\sigma$-algebra generated by $\PAR{Z(s)}_{s\in [0,t]}$.

\medskip

{\it Related literature.}
Assessing the likelihood of ruin and its potential impact has long been a fundamental topic in the classical mathematical theory of insurance. For a thorough overview of the subject, the reader may consult \cite{Asmussen-Albrecher10}, while additional comprehensive treatments can be found in \cite{Dickson2017,Schmidli2017}. 

In the literature, the surplus process is often modeled using a classical Cramér–Lundberg process, in which claims arrive according to a Poisson process and claim sizes are subtracted from accumulated premiums. In a more general setting, the surplus process can be a spectrally negative Lévy process (SNLP), while a \textit{dual process} can be a spectrally positive Lévy process. In the former case, the negative jumps of the SNLP correspond to individual claims, while the continuous part captures more regular and frequent fluctuations in the reserve. In many applications, particularly when modeling the aggregate risk of a large number of relatively small policyholders, the surplus process is further approximated by a Brownian motion, which can be interpreted as a diffusion limit of many small and frequent claims.

Over the past decades, insurance models with underlying dynamics corresponding to epidemic models have received increasing interest. Without attempting to give a full account of the area, two branches of the literature are particularly worth mentioning in the context of the present paper:
\begin{itemize}
    \item[$\circ$] First, there are papers that rely on the framework of \emph{Markov additive processes} \cite{Lefevre-Simon-SIR21,Lefevre-Simon-SIS22}. In their main example, the number of infected individuals (in a population of fixed size) is modeled by a stochastic SIS model, which can be represented as a continuous-time Markov chain of birth–death type. The claim arrival processes of infected and non-infected individuals differ and are typically chosen to be of L\'evy type. This naturally places the model within the class of MAP-driven ruin models \cite[Chapter VII]{Asmussen-Albrecher10}; see also the general treatment in \cite{vanKreveldMandjesDorsman2022} and the corresponding tail asymptotics in the light-tailed regime discussed in \cite{vanKreveldMandjesDorsman2024}. It is also noted that there is a direct connection with a dual queueing model; see, e.g., the analysis in \cite{DiekerMandjes2011}.
    \item[$\circ$] 
    Second, there are papers that explicitly model insurance risk processes driven by two intrinsically different sub-populations. In \cite{Li-Garrido05}, the focus is on a two-class model with $Z_t:=u+\rc t+X_t+Y_t $,
where $X$ and $Y$ are, respectively, a compound Poisson process and a compound Sparre-Andersen process with Erlang(2) claim inter-arrival times; $u$ denotes the initial reserve level and $\rc$ the constant premium rate. The ruin probability is studied in this setting. In \cite{Ji-Zhang10}, the focus is on the substantially more general setup where $X$ and $Y$ are taken to be independent compound renewal processes.
\end{itemize}

A book on epidemic models in insurance \cite{Pandemic-Insurance-book21} has been recently published, in which we recover the model introduced in \cite{Feng-Garrido11} and the approach of \cite{Lefevre-Simon-SIR21,Lefevre-Simon-SIS22}. Importantly, the two models we propose in the present paper have the MAP-driven ruin model, discussed above, as a special case. 

\medskip

{\it Contributions.} The main contributions of this work can be summarized as follows.
\begin{itemize}
\item[$\circ$] We introduce two types of models, referred to as the {\it collective model} and the {\it individual model}, designed to describe risk processes in a setting where the number of individuals in each sub-population evolves stochastically over time. These general classes of models allow to explicitly incorporate the effects of population dynamics when assessing the risk of insurance portfolios.

\item[$\circ$] We study scaling limits for both models by establishing a law of large numbers and a central limit theorem for the long-time behavior of these processes. This analysis provides insight into the macroscopic behavior of the aggregate risk process induced by the population dynamics.

\item[$\circ$] We derive bounds and approximations for the ruin probability, providing practical tools for actuarial risk assessment in populations with dynamically changing sub-populations.

\item[$\circ$] Finally, we illustrate the proposed framework through several examples involving SIS epidemic dynamics and different classes of risk processes.
\end{itemize}

\medskip

{\it Organization.} The paper is organized as follows. In Section~\ref{sec:Two risk models}, we define our collective and individual models. Section~\ref{sec: long time behavior} describes the long-time behavior of each model through laws of large numbers and central limit theorems, whereas Section~\ref{sec:ruin probability} focuses on ruin probabilities and tail asymptotics. Finally, in Section~\ref{sec:SIS}, we specialize our models using  SIS epidemic dynamics.

\section{Two risk models}\label{sec:Two risk models}

In this paper, in the spirit of classical actuarial risk theory, we introduce two new risk models: a \textit{collective model} based on $(X,Y,F)$ and an \textit{individual model} based on $(X,Y,I)$. We assume throughout that all models considered are defined as specified in the itemized list describing the \emph{stochastic objects} in the introduction.

\subsection{A collective model}\label{sec:introduction collective model}

First, we define a \textit{collective model}. To this end, let $X$ (resp.\ $Y$) denote the surplus process of an insurance portfolio with risks of type $\rA$ (resp.\ type $\rB$). Since the fraction of risks of type $\rA$ at time $t$ is given by $F(t)$, the aggregated surplus risk process $V=\PAR{V(t)}_{t \geq 0}$ is given, for each $t \geq 0$, by
\begin{equation}\label{eq:def-V}
V(t) :=  \int_{0+}^t F(s-) \, \rd X(s) + \int_{0+}^t \PAR{1-F(s-)} \, \rd Y(s) .
\end{equation}
The processes $F$ and $1-F$ act as dynamic weights between the two classes of risks. In particular, if $F$ is constant, i.e., if $F(t)=F_0$ for all $t \geq 0$ for some constant $F_0 \in [0,1]$, then $V(t) =   F_0 X(t)  + \PAR{1-F_0}  Y(t)$ so that $V$ is a (static) convex combination of the processes $X$ and $Y$.

\begin{lemma}\label{lem:characteristic function V}
The characteristic function of $V$ is given, for $t\geq 0$ and $\alpha \in \dR$, by
\begin{equation}\label{eq:characteristic function V}
\dE\SBRA{\re^{i \alpha \,{V(t)}}} = \dE\SBRA{\exp \left({\int_0^t \varphi_X \PAR{\alpha F(s)} \, \rd s + \int_0^t \varphi_Y \PAR{\alpha (1-F(s))} \,\rd s} \right)} .
\end{equation}
\end{lemma}
\begin{proof}
Fix $t\geq 0$ and $\alpha \in \dR$. Let $n \in {\mathbb N}$ and define $t_k=\frac{k}{n}t$ for each $k\in\BRA{0,\ldots,n}$. Since $F$ is a bounded càdlàg process, by \cite[Theorem II.21]{protter04}, we have
\begin{align*}
\int_{0+}^t F(s-) \, \rd X(s) = \lim_{n \to \infty} \sum_{k=0}^{n-1} F(t_k) \PAR{X(t_{k+1})-X(t_k)} ,
\end{align*}
where the limit is \textit{uniformly on compacts in probability}. As a consequence, we have
\[
\dE \SBRA{\re^{i\alpha\int_{0+}^t F(s-)\rd X(s)}} = \lim_{n \to \infty} \dE \SBRA{\exp \left(i \alpha \sum_{k=0}^{n-1} F(t_k) \PAR{X(t_{k+1})-X(t_k)} \right)}.
\]
Since $X$ is a Lévy process, the increments $\BRA{X(t_{k+1})-X(t_k) , 1 \leq k \leq n}$ are independent and all distributed as the random variable $X(t/n)$. Of course, similar statements can be made about $Y$.

Using the independence between $X$, $Y$ and $F$, we can write
\begin{align*}
&\hspace{-0.7cm}\dE\SBRA{\exp\PAR{i\alpha\sum_{k=0}^{n-1}F(t_k)\PAR{X(t_{k+1})-X(t_k)}+i\alpha\sum_{k=0}^{n-1}\PAR{1-F(t_k)}\PAR{Y(t_{k+1})-Y(t_k)}}\,\Big\vert \,\cF^F_t}\\
=\:&\prod_{k=0}^{n-1}\dE\SBRA{\exp\left({i\alpha F(t_k)\PAR{X(t_{k+1})-X(t_k)}}\right)\,\Big\vert \,\cF^F_t} \\
& \qquad \times \prod_{k=0}^{n-1}\dE\SBRA{\exp\left({i\alpha\PAR{1-F(t_k)}\PAR{Y(t_{k+1})-Y(t_k)}}\right)\,\Big\vert\, \cF^F_t}\\
=\:&\prod_{k=0}^{n-1}\exp\left({\varphi_X\PAR{\alpha F(t_k)}\PAR{t_{k+1}-t_k}}\right)
\,\prod_{k=0}^{n-1}\exp\left({\varphi_Y\PAR{\alpha\PAR{1-F(t_k)}}\PAR{t_{k+1}-t_k}}\right)\\
=\:&\exp\PAR{\sum_{k=0}^{n-1}\varphi_X\PAR{\alpha F(t_k)}\PAR{t_{k+1}-t_k}
+\sum_{k=0}^{n-1}\varphi_Y\PAR{\alpha\PAR{1-F(t_k)}}\PAR{t_{k+1}-t_k}}.
\end{align*}

Since $\varphi_X$ and $\varphi_Y$ are continuous functions, and since $F$ is càdlàg, then we have, almost surely,
\[
\lim_{n \to \infty} \sum_{k=0}^{n-1} \varphi_X \PAR{\alpha F(t_k)} \PAR{t_{k+1}-t_k} = \int_0^t \varphi_X \PAR{\alpha F(s)} \, \rd s .
\]
Again, the same argument applies for the integral with $\varphi_Y$. Finally, taking expectations and using the dominated convergence theorem completes the proof.
\end{proof}

\subsection{An individual model}\label{sec:introduction individual model}

We now define an \textit{individual model}. To this end, we consider a population of $N$ individuals and introduce $\chi=(\chi(t))_{t\geqslant 0}$, a multivariate càdlàg jump process describing the state of each individual. More precisely, we have $\chi(t)=\PAR{\chi^1(t), \dots, \chi^N(t)}$ in which $\chi^j(t) \in\{\rA,\rB\}$ represents the state of the $j$-th individual at time $t$. Then define the number of individuals of type $\rA$ at time $t$ by $I(t):=\sum_{j=1}^N \ind_\BRA{\chi^j(t)=\rA}$.

In this model, let $X$ denote a generic risk process for an individual of type $\rA$, and $Y$ a generic risk process for an individual of type $\rB$. Let $X^1,\ldots,X^N$ be i.i.d.\ processes that are distributed as $X$, and let $Y^1,\ldots,Y^N$ be i.i.d.\ processes that are distributed as $Y$ (and that are independent of $X^1,\ldots,X^N$). Finally, let us define the risk process corresponding to the $j$-th individual by
\begin{align}\label{eq:def W^j}
    W^j(t)&:=\int_{0+}^t\ind_\BRA{\chi^j(s-)=\rA}\rd X^j(s)+\int_{0+}^t\ind_\BRA{\chi^j(s-)=\rB}\rd Y^j(s).
\end{align}
The central object of interest is the aggregated risk process
\begin{equation}\label{eq:def-W}
W(t):= \sum_{j=1}^{N}W^j(t) .
\end{equation}

We refer to this model as the {\it individual} model, as it distinguishes the contributions from each individual within the population to the overall risk. The individual risk processes $(W^j)_{1 \leq j \leq N}$ are not assumed to be independent, but independent conditional on $\chi$.

\begin{remark}
For each $j$, as $\ind_\BRA{\chi^j(t)=\rA}=1-\ind_\BRA{\chi^j(t)=\rB}$, we note that $W^j$, defined in~\eqref{eq:def W^j}, falls within the framework of the collective model with $F^j(s)=\ind_\BRA{\chi^j(s)=\rA}$.
\end{remark}

\begin{lemma}\label{lem:characteristic function W}
The characteristic function of the risk process $W$ is given by
\begin{align}\label{eq:characteristic function W}
\dE\SBRA{\re^{i\alpha\, W(t)}}&=\dE\SBRA{\exp\left({\varphi_X(\alpha )\int_0^tI(s)\,\rd s+\varphi_Y(\alpha )\int_0^t(N-I(s))\,\rd s}\right)} \notag\\
&=\re^{\varphi_Y(\alpha)Nt}\,\dE\SBRA{\exp\left({ \PAR{\varphi_X(\alpha )-\varphi_Y(\alpha)}\int_0^tI(s)\,\rd s}\right)}. 
\end{align}
\end{lemma}
\begin{proof}
We deduce  the characteristic function of $W$ from Lemma~\ref{lem:characteristic function V}. Conditional on $\cF^\chi_t$, which is the $\sigma$-algebra generated by $(\chi(s))_{s\in[0,t]}$, from the proof of Lemma~\ref{lem:characteristic function V}, by the independence between $X^j,Y^j,\chi$, we have
\[
\dE\SBRA{\re^{i\alpha W^j}\,\Big\vert \,\cF_t^\chi}=\exp\left({\int_0^t\varphi_X\PAR{\alpha \ind_\BRA{\chi^j(s)=\rA}}\, \rd s+\int_0^t\varphi_Y\PAR{\alpha \PAR{1-\ind_\BRA{\chi^j(s)=\rA}} }\,\rd s}\right).
\]
As $\varphi_X(0)=0=\varphi_Y(0)$, we can write
\[
\dE\SBRA{\re^{i\alpha \,W^j}\,\Big\vert \,\cF_t^\chi}=\exp\left({\varphi_X\PAR{\alpha }\int_0^t\ind_\BRA{\chi^j(s)=\rA}\,\rd s+\varphi_Y\PAR{\alpha}\int_0^t \PAR{1-\ind_\BRA{\chi^j(s)=\rA}} \,\rd s}\right).
\]
As $(W^j)_{1 \leq j \leq N}$ on $[0,t]$ are conditionally independent given $\mathcal{F}_t^\chi$, we thus have
\begin{align*}
\dE\SBRA{\re^{i\alpha \,W}}&=\dE\SBRA{\prod_{j=1}^N\exp\left({\varphi_X\PAR{\alpha }\int_0^t\ind_\BRA{\chi^j(s)=\rA}\,\rd s+\varphi_Y\PAR{\alpha}\int_0^t \PAR{1-\ind_\BRA{\chi^j(s)=\rA}}\, \rd s}\right)}\\
&=\dE\SBRA{\left({\varphi_X\PAR{\alpha }\int_0^tI(s)\,\rd s+\varphi_Y\PAR{\alpha}\int_0^t \PAR{N-I(s)}\,\rd s}\right)} .
\end{align*}
\end{proof}

It can now be seen that $W$ has the same distribution as the sum of time-changed versions of the risk processes $X$ and $Y$. To this end, define the process $U=(U(t))_{t\ge 0}$, for any $t\geq 0$, by 
\[
U(t):=X\left(J(t)\right)+Y\left(Nt-J(t)\right) \:\:\text{ with } J(t):=\int_0^tI(s)\,\rd s.
\]

\begin{proposition}\label{prop:equality in distribution of $W$} 
     The processes $W$ and $U$ have the same distribution. 
    \end{proposition}
This proposition is proved in Appendix~\ref{A:proof equality in distribution of $W$}.

\begin{remark}\em
When $I$ is a birth-and-death process, $W$ can also be written as a double sum. To this end, introduce $(T_k)_{k\geq 0}$ the successive jump times of the birth-and-death process and let $K_t$ be the number of jumps until time $t$. Using the notations $X^j(s;t):=X^j(t)-X^j(s)$, $Y^j(s;t):=Y^j(t)-Y^j(s)$, and $1_j(t):=\ind_{\BRA{\chi_N^j(t)=\rA}}$, we can alternatively write
\begin{align}
    W(t)=\:\sum_{k=0}^{K_t-1} &\Bigg(\sum_{j=1}^{N}X^j(T_k;{T_{k+1}})1_j(T_k)+\sum_{j=1}^{N}Y^j(T_k;{T_{k+1}})(1-1_j(T_k))\Bigg)\,+\notag
    \\&\Bigg(\sum_{j=1}^{N}X^j({T_{K_t}};t)1_j(T_{K_t})+\sum_{j=1}^{N}Y^j({T_{K_t}};t)(1-1_j(T_{K_t}))\Bigg).\label{eq:WNT}
\end{align}
It is worth noting that the above expression for $W(t)$ can be used to provide a direct proof of the characteristic function of $W$, as stated in Lemma~\ref{lem:characteristic function W}.\hfill$\Diamond$
\end{remark}

\subsection{Comparison of the two models}
Having introduced the collective and individual models, we can shed more light on their relationship. For this section only, we denote by $X^\rc$ and $Y^\rc$ the L\'evy processes in the collective model, and by $X^\ri$ and $Y^\ri$ their counterparts in the individual model. From the respective characteristic functions of the risk processes $V$ and $W$, we note that they do not have the same distribution, but we can easily compare their means and variances.

\begin{lemma}\label{lem:esp-variance}
Assume that $X^\rc, X^\ri$ and $Y^\rc, Y^\ri$ have  finite second moments. Then, 
\begin{align*}
&\dE[V(t)]=\dE[X^\rc(1)]\int_0^t\dE[F(s)]\,\rd s+\dE[Y^\rc(1)]\int_0^t\dE[1-F(s)]\,\rd s,\\
&\var(V(t))=\var(X^\rc(1))\, \int_0^t \dE\SBRA{F^2(s)}\,\rd s+\var(Y^\rc(1))\,\int_0^t\dE\SBRA{\PAR{1-F(s)}^2}\rd s\\
&\hskip 2.5cm+\PAR{\dE[X^\rc(1)]-\dE[Y^\rc(1)]}^2\var\PAR{\int_0^t F(s)\,\rd s},
\\
&\dE[W(t)]=\dE[X^\ri(1)]\int_0^t\dE[I(s)]\rd s+\dE[Y^\ri(1)]\int_0^t\dE[N-I(s)]\,\rd s,\\
&\var(W(t))=
\var(X^\ri(1))\,\int_0^t \dE\SBRA{I(s)}\,\rd s+\var(Y^\ri(1))\,\int_0^t\dE\SBRA{N-I(s)}\,\rd s,\\
&\hskip 2.5cm+\PAR{\dE[X^\ri(1)]-\dE[Y^\ri(1)]}^2\var\PAR{\int_0^tI(s)\,\rd s}.
\end{align*}
\end{lemma}

A proof of this last lemma is provided in Appendix~\ref{B:proof means and variances}.

\begin{remark}\em 
Taking into account the size of the population $N$, recalling that $F(t)={I_N(t)}/{N}$, we can argue that both models have the same mean, but the collective model has a smaller variance. Indeed, for the mean we have 
\begin{align*}
\dE[V(t)]&=\frac{1}{N}\dE[X^\rc(1)]\int_0^t\dE[I(s)]\,\rd s+\frac{1}{N}\dE[Y^\rc(1)]\int_0^t\dE[N-I(s)]\,\rd s.
\end{align*}
Since $F^2\leq F$ and $(1-F)^2\leq 1-F$, we easily observe that 
\begin{align*}
\var(V(t))&\leq \frac{1}{N}\var(X^\rc(1))\,\int_0^t \dE\SBRA{I(s)}\rd s+\frac{1}{N}\var(Y^\rc(1))\,\int_0^t\dE\SBRA{N-I(s)}\rd s\\
&\hskip 1.5cm+\PAR{\frac{\dE[X^\rc(1)]-\dE[Y^\rc(1)]}{N}}^2\var\PAR{\int_0^tI(s)\,\rd s}.
\end{align*}
The latter inequality reflects that the individual model has a higher degree of \textit{intrinsic variability}, compared to the collective model. 

Note that when there is only one class, i.e., when $F\equiv 1$ and $I\equiv N$, the risk processes $V$ and $W$ are the same if we assume that the collective risk process is the sum of independent individual risk processes.  \hfill$\Diamond$
\end{remark}
A modeller with detailed knowledge of the portfolio is best placed to decide whether an individual or a collective model is more appropriate. In particular, if a portfolio is small, comprising only a few insureds/risks or with data available at the individual level, it seems reasonable to expect that an individual model will be more appropriate, as each individual has a more significant impact on the aggregated claims. Larger, more homogeneous portfolios, on the other hand, fit more naturally within a collective model.

\medskip
The variety of processes $F$ and $I$ allowed within this framework is vast. In the following example, we discuss an important Markovian subclass. We emphasize, however, that the framework is by no means restricted to such Markovian processes.

\begin{exple}\em 
 An important example for the process $F$ is the following. Consider the collective model in which $F$ attain values in $\{0,1/N,2/N,\ldots,1\}$ and let it jump between those values according to an $(N+1)$-dimensional irreducible continuous-time Markov chain, with transition rate matrix $\smash{Q=(Q_{k\ell})_{k,\ell=0}^N}$ (i.e., $Q_{k\ell}\geq 0$ for $k\not=\ell$ and row sums are $0$). The characteristic function of $V(t)$ can be identified as follows. Define, for given $ 0\leq k,\ell\leq N$ and $\alpha\in{\mathbb R}$,
\[\Psi^\rc_\ell(t):= {\mathbb E}\left[{\rm e}^{i\alpha\,V(t)}\ind_{\{F(t)=\ell/N\}}\,\Big|\,F(0)=k/N\right].\] 
Then by using Lemma \ref{lem:characteristic function V} it is readily verified, setting up Kolmogorov equations, that, as $\Delta t\downarrow 0$,
\[\Psi^\rc_{\ell}(t+\Delta t)= \sum_{m\not=\ell} \Psi^\rc_m(t)\,Q_{m\ell}\,\Delta t+\Psi^\rc_\ell(t)\,(1+Q_{\ell\ell}\,\Delta t)\exp\big(\Delta t\,\psi^\rc_\ell(\alpha)\big),\]
with $\psi_\ell^\rc(\alpha):=\varphi_X(\alpha\,\ell/N)+\varphi_Y(\alpha\,(1-\ell/N))$.
Upon expanding the exponential, subtracting $\Psi^\rc_{\ell}(t)$ from both sides, dividing by $\Delta t$ and sending $\Delta t$ to $0$ yields a system of linear differential equations:
\begin{equation}\frac{\rm d}{{\rm d}t}\Psi^\rc_\ell(t)=\sum_{m=0}^N \Psi^\rc_m(t)\,Q_{m\ell}+\Psi^\rc_\ell(t)\,\psi^\rc_\ell(\alpha), \label{eq:syst de}\end{equation}
with evident initial conditions. 
It is now instructive to consider a similar setup in the individual model, where $I$ is an irreducible continuous-time Markov chain on $\{0,\ldots,N\}$, with transition rate matrix $\smash{Q=(q_{k\ell})_{k,\ell=0}^N}$. Define, for a given $k,\ell\in \{0,\ldots,N\}$ and $\alpha\in{\mathbb R}$,
\[\Psi^\ri_\ell(t):= {\mathbb E}\left[{\rm e}^{i\alpha\,W(t)}\ind_{\{I(t)=\ell\}}\,\Big|\,I(0)=k\right]\]
and $\psi_\ell^\ri(\alpha):=\varphi_X(\alpha)\,\ell+\varphi_Y(\alpha)\,(N-\ell)$. 
It is straightforward to check, applying Lemma \ref{lem:characteristic function W}, that we obtain the same system of linear differential equations as in \eqref{eq:syst de}, but with the superscript c replaced by i.  
Observe that the main difference between the collective model and the individual model lies in the distinction between the functions $\psi_\ell^\mathrm{c}(\alpha)$ and $\psi_\ell^\mathrm{i}(\alpha)$.
 \hfill$\Diamond$  
\end{exple}

\section{Long time behavior}\label{sec: long time behavior}

In this section, we investigate the long-time behavior of the processes $V$ and $W$, as defined in \eqref{eq:def-V} and \eqref{eq:def-W}, respectively. Specifically, for each of these processes, we establish both a \textit{law of large numbers} and a \textit{central limit theorem}. In what follow, `$\smash{\overset{\rm {\small a.s.}}{\longrightarrow}}$' denotes almost sure convergence, and `$\smash{\overset{\rm {\small d}}{\longrightarrow}}$' denotes convergence in distribution.

Throughout, we again assume that the models under consideration are defined as specified in the 
\emph{Stochastic objects} paragraph of the introduction.

\subsection{Asymptotic behavior in the collective model}

The main objective of this subsection is to  prove a strong law of large numbers ({\sc slln}) and a central limit theorem ({\sc clt}) for the process $V$.
The following assumption is in place.
\begin{assumption}\label{A:hyp LLN for V}
There exists a constant $\rho\in[0,1]$ such that 
\begin{equation}\label{eq:def-rho}
 \frac{1}{t}\int_{0}^{t}F(s)\, \rd s  \overset{\rm {\small a.s.}}{\underset{t\to \infty}{\longrightarrow}}\rho.
\end{equation}
\end{assumption}
Note that, if $F$ converges almost surely to a (nonnegative) constant or if $F$ is an ergodic Markov process, then Assumption~\ref{A:hyp LLN for V} is satisfied. 

\begin{proposition}\label{prop:LLN for stochastic integral}
Under Assumption~\ref{A:hyp LLN for V}, if $X(1)$ has a finite second moment, then
\begin{equation*}
\frac{1}{t} \int_{0+}^{t} F(s-) \,\rd X(s) \,\overset{\rm {\small a.s.}}{\underset{t\to \infty}{\longrightarrow}} \rho\,\dE[X(1)].
\end{equation*}
\end{proposition}
Note that the result is still true when $\rho$ is an $\cF^F$-measurable random variable.

\begin{proof}
We note that $\PAR{M(t)}_{t \geq 0}$, given by
\begin{equation}
M(t) := \int_{0+}^t F(s-)\,\rd X(s) - \dE[X(1)] \int_0^t F(s)\,\rd s ,
\end{equation}
is a square-integrable, right-continuous martingale. Indeed, $M$ is $\PAR{\cF_t}_{t\geq 0}$-adapted, with finite second moments, and
\begin{align*}
\dE[M_{t+s}\,\vert\,\cF_t]&=M_t+\dE\SBRA{\int_{t+}^{t+s}F(u-)\,\rd X(u)\,\Big\vert \,\cF_t}-\dE[X(1)]\,\dE\SBRA{\int_t^{t+s}F(u)\,\rd u\,\Big\vert\, \cF_t}=M_t.
\end{align*}
To obtain the last equality, we use the independence of $F$ and $X$, together with the independence of the increments of $X$, which yields
\begin{align*}\displaystyle{\dE\SBRA{\int_{t+}^{t+s}F(u-)\,\rd X(u)\,\Big\vert \,\cF^F\vee \cF_t}=\dE[X(1)]\int_t^{t+s}F(u)\,\rd u}. \end{align*}

By Doob's maximal inequality, for $0 \leq r < t$, we can write:
\begin{equation}
\dE \SBRA{\sup_{r \leq s \leq t} \PAR{\frac{M(s)}{s}}^2} \leq \frac{1}{r^2} \dE \SBRA{\sup_{r \leq s \leq t} M(s)^2} \leq \frac{4}{r^2}\sup_{r\leq s\leq t} \dE \SBRA{M(s)^2} ,
\end{equation}
where $\dE \SBRA{M(s)^2} = \var(X(1)) \,\dE \SBRA{\int_0^s F(u)^2 \,\rd u} \leq t\, \var(X(1))$ for $s\in[r,t]$. 
Fix $\delta>0$ and $n \geq 1$. Upon combining the above inequalities, and in addition using Markov's inequality, we conclude
\begin{equation}
\dP \PAR{\sup_{2^n \leq t \leq 2^{n+1}} \ABS{\frac{M(t)}{t}} > \delta} \leq \frac{1}{\delta^2}{\dE \SBRA{\sup_{2^n \leq t \leq 2^{n+1}} \PAR{\frac{M(t)}{t}}^2}}  \leq \frac{4 \var(X(1))}{\delta^2} \frac{1}{2^{n-1}} .
\end{equation}
Now observe that
\begin{align*}
    \sum_{n \geq 1} \dP \PAR{\sup_{2^n \leq t \leq 2^{n+1}} \ABS{\frac{M(t)}{t}} > \delta} < \infty,
    \end{align*}
    so that by the first Borel-Cantelli Lemma we deduce that 
    \begin{align*}
    \dP \PAR{\limsup_{n\to\infty} \BRA{\sup_{2^n \leq t \leq 2^{n+1}} \ABS{\frac{M(t)}{t}} >\delta}}=1.\end{align*} As $\delta$ is arbitrary, the result follows.
\end{proof}

Here is a direct consequence in the collective model.
\begin{corollary}[{\sc slln} for $V$]\label{coro:LLN for V}
Under Assumption~\ref{A:hyp LLN for V}, if $X(1)$ and $Y(1)$ have finite second moments, then 
\[
\frac{1}{t} V(t) \overset{\rm { a.s.}}{\underset{t\to \infty}{\longrightarrow}} \rho\,\dE[X(1)]+(1-\rho)\,\dE[Y(1)].
\]
\end{corollary}

In order to prove the {\sc clt}, we impose a second assumption.
\begin{assumption}\label{A:hyp CLT for V}
There exists a constant
$\eta\in(0,1]$ such that
\begin{equation}\label{eq:def-eta}
  \frac{1}{t} \int_{0}^{t}F(s)^2\rd s \overset{\rm {\small a.s.}}{\underset{t\to \infty}{\longrightarrow}}\eta^2.
\end{equation}
Also, with  $Z^{\rm c}$  a Gaussian random variable with mean $0$ and variance $\sigma_{\rm c}^2$, 
\[
\PAR{\frac{1}{\sqrt t} \int_{0}^{t} \left(F(s) -\dE[F(s)]\right) \rd s} \overset{\rm d}{\underset{t \to \infty}{\longrightarrow}} Z^{\rm c} .
\]
\end{assumption}

The following result gives the classical speed of convergence in Proposition~\ref{prop:LLN for stochastic integral}. Let ${\dZZ}^{\rm c} $ be a Gaussian random variable with mean zero and variance given by
\[
\var(\dZZ^{\rm c}  )=\PAR{\var(X(1))+\var(Y(1))}\,\eta^2+\var(Y(1))\,(1-2\rho)+\PAR{\dE[X(1)]-\dE[Y(1)]}^2\,\sigma_{\rm c}^2,
\]
where $\rho$, $\eta$ and $\sigma_{\rm c}$ are as given in Assumptions~\ref{A:hyp LLN for V} and~\ref{A:hyp CLT for V}.

\begin{proposition}[{\sc clt} for $V$]\label{prop:CLT for V} 
Under Assumptions~\ref{A:hyp LLN for V} and~\ref{A:hyp CLT for V}, if $X(1)$ and $Y(1)$ have finite second moments, then
\begin{equation}
\frac{1}{\sqrt{t}}\PAR{V(t)-\dE[V(t)]}\overset{\rm d}{\underset{t\to \infty}{\longrightarrow}}\dZZ^{\rm c}   .
\end{equation}
\end{proposition}

\begin{proof} We prove this result by showing that the characteristic function of $\frac{1}{\sqrt{t}}\PAR{V(t)-\dE[V(t)]}$ converges to that of $\dZZ^{\rm c}  $. Throughout this proof we keep $\alpha\in\dR$ fixed.
By Lemma~\ref{lem:characteristic function V}, we have
\begin{align*}
m_\alpha(t):=\dE\SBRA{\re^{\frac{i\alpha}{\sqrt{t}}\PAR{V(t)-\dE[V(t)]}}}&=\dE\SBRA{\exp\left({\int_0^t\varphi_X\PAR{\frac{\alpha}{\sqrt{t}} F(s)}\rd s+\int_0^t\varphi_Y\PAR{\frac{\alpha}{\sqrt{t}}(1- F(s))}\rd s}\right)}\\
&\hskip -0.044cm \times\exp\left({-\frac{\alpha}{\sqrt{t}}\PAR{\varphi_X'(0)\int_0^t\dE[F(s)]\rd s+\varphi_Y'(0)\int_0^t\dE[1-F(s)]\rd s}}\right).
\end{align*}
Using the Taylor expansion, we have that as $t\to\infty$, for any fixed $s\geq 0$,
\begin{align*}
\varphi_X\PAR{\frac{\alpha}{\sqrt{t}} F(s)}&=\frac{\alpha}{\sqrt{t}} F(s)\varphi'_X(0)+\frac{\alpha^2}{2t} F(s)^2\varphi''_X(0)+\varepsilon_X(s,t)\\
\varphi_Y\PAR{\frac{\alpha}{\sqrt{t}} (1-F(s))}&=\frac{\alpha}{\sqrt{t}} (1-F(s))\varphi'_Y(0)+\frac{\alpha^2}{2t} (1-F(s))^2\varphi''_Y(0)+\varepsilon_Y(s,t)
\end{align*}
with $\varepsilon_X, \varepsilon_Y$ denoting $\cF^F$-measurable random functions, such that $\lim_{t\to \infty}t\varepsilon_X(s,t)=0$ and $\lim_{t\to \infty}t\varepsilon_Y(s,t)=0$ uniformly in $F$ and $s\in\dR_+$ (where it is recalled that $F$ has values in $[0,1]$).
Consequently, with  $\varepsilon(t):=\int_0^t\PAR{\varepsilon_X(s,t)+\varepsilon_Y(s,t)}\,\rd s$,
\begin{align*}
m_\alpha(t)&=\dE\Bigg[
\exp\left({\frac{\alpha}{\sqrt{t}}\PAR{\varphi'_X(0)-\varphi_Y'(0)}\int_0^t\PAR{F(s)-\dE[F(s)]}\rd s
}\right)\\&\quad\quad\exp\left({\frac{\alpha^2}
{2t}\int_0^t\PAR{\varphi_X''(0)F(s)^2+\varphi_Y''(0)(1-F(s))^2}\rd s+\varepsilon(t)}\right)\Bigg]\\
&=\dE\Bigg[
\exp\left({i\frac{\alpha\PAR{\dE[X(1)]-\dE[Y(1)]}}{\sqrt{t}}\int_0^t\PAR{F(s)-\dE[F(s)]}\rd s
}\right)\\&\quad\quad\exp\left({-\frac{\alpha^2}
{2t}\int_0^t\PAR{\var(X(1))F(s)^2+\var(Y(1))(1-F(s))^2}\rd s+\varepsilon(t)}\right)\Bigg].
\end{align*}
By the assumptions imposed, we have, with $\bar\sigma^2:=\PAR{\var(X(1))+\var(Y(1))}\eta^2+\var(Y(1))(1-2\rho)$,
\begin{align*}
&H(t):=\frac{1}{\sqrt t}\int_{0}^{t}\left(F(s) -\dE[F(s)]\right)\rd s\overset{\rm d}{\underset{t\to \infty}{\longrightarrow}}Z^{\rm c}\\
&G(t):=\frac{1}
{t}\int_0^t\PAR{\var(X(1))F(s)^2+\var(Y(1))(1-F(s))^2}\,\rd s\overset{\rm {\small a.s.}}{\underset{t\to \infty}{\longrightarrow}}
\bar\sigma^2.
\end{align*}
Consequently, noting that $\var(\dZZ^{\rm c}  )=\bar\rho^2\sigma_{\rm c}^2+\bar\sigma^2$ with $\bar\rho:=\dE[X(1)]-\dE[Y(1)]$,
\begin{align*}
\ABS{m_\alpha(t)-\re^{-\frac{\alpha^2}{2}\var({\mathscr Z}^{\rm c})}}&=\ABS{\dE\SBRA{\re^{i\alpha\,\bar\rho\,H(t)}\re^{-\frac{\alpha^2}{2}G(t)+\varepsilon(t)}}-\re^{-\frac{\alpha^2}{2}\var({\mathscr Z}^{\rm c})}}\\
&\leq \dE\SBRA{\ABS{\re^{-\frac{\alpha^2}{2}G(t)+\varepsilon(t)}-\re^{-\frac{\alpha^2\bar\sigma^2}{2}}}}+\re^{-\frac{\alpha^2\bar\sigma^2}{2}}\,\ABS{\dE\SBRA{\re^{i\alpha\,\bar\rho\,H(t)}}-\re^{-\frac{\alpha^2}{2}\bar\rho^2\sigma_{\rm c}^2}}.
\end{align*}
The first term in the last expression of the previous display goes to $0$ as $t\to \infty$ by the dominated convergence theorem, while  the second term vanishes by virtue of Assumption~\ref{A:hyp CLT for V}. The conclusion follows. 
\end{proof}

Observe that the expression for $\var(\dZZ^{\rm c}  )$ is in line with the limit of $\var(V(t))/t$ as $t\to\infty$, with $\var(V(t))$ as given in Lemma \ref{lem:esp-variance}.

\subsection{Asymptotic behavior of the individual model}\label{sec: individual model}

In this subsection, we derive the counterparts of the results obtained for the collective model in the previous subsection, now for the individual model. More specifically, we establish a {\sc slln} and a {\sc clt} for the process $W$ defined in \eqref{eq:def-W}. In this model, 
\begin{align*}
   W(t)= \sum_{j=1}^{N}W^j(t), \quad W^j(t)=\int_0^t\ind_\BRA{\chi^j(s^-)=\rA}\rd X^j(s)+\int_0^t\ind_\BRA{\chi^j(s^-)=\rB}\rd Y^j(s).
\end{align*}

From the results on the collective model, we straightforwardly deduce a {\sc slln} for the individual model.
 
 \begin{assumption}\label{A:hyp LLN for W}
 There exists a (deterministic)
 vector $\PAR{\rho^1,\ldots,\rho^N}$ in $[0,1]^N$ such that
     \[
    \lim_{t\to\infty}\frac{1}{t}\int_{0}^{t}\ind_\BRA{\chi^j(s)=\rA} \rd s \overset{\rm {\small a.s.}}{\underset{t\to \infty}{\longrightarrow}}\rho^j.
     \]
 \end{assumption}
 Note that when $\chi$ is an ergodic Markov process, the above limits are well defined. We can now state our {\sc slln}; recall that $I(t)$ denotes the number of individuals of type $\rA$ at time $t$.
 
\begin{proposition}[{\sc slln} for $W$]\label{prop:LLN for W}
   Under Assumption~\ref{A:hyp LLN for W}, if
    $X(1)$ and $Y(1)$ have finite second moments, then,
    \[
    \frac{1}{t} W(t) \overset{\rm {\small a.s.}}{\underset{t\to \infty}{\longrightarrow}}\rho \,\dE[X(1)]+(N-\rho)\, \dE[Y(1)],
    \]
    where $\rho:=\sum_{j=1}^N\rho^j=\lim_{t\to\infty}\frac{1}{t}\int_{0}^{t}I(s) \,\rd s$ almost surely.
\end{proposition}
Note that the result is still true when $\PAR{\rho^1,\ldots,\rho^N}$ is an $\cF^F$-measurable random vector.

\begin{proof}
As the processes $(X^j)_{1\leq j\leq N}$, $(Y^j)_{1\leq j\leq N}$ and $\chi_N$ are independent, under Assumption~\ref{A:hyp LLN for W} each individual risk process $W^j$
satisfies the conditions of Corollary~\ref{coro:LLN for V}. Hence, for any $j\in\BRA{1,\ldots,N}$, 
\[
\frac{W^j(t)}{t} \overset{\rm {\small a.s.}}{\underset{t\to \infty}{\longrightarrow}}\rho^j\,\dE[X(1)]+(1-\rho^j)\,\dE[Y(1)].
\]
Taking the sum over $j$, we obtain the result since $I(t)=\sum_{j=1}^N\ind_\BRA{\chi^j(t)=\rA}$.
\end{proof}

So as to prove a {\sc clt} for $W$, we introduce  the following  assumption.
\begin{assumption}\label{A:hyp TCL for W}
With $Z^\ri$ a Gaussian random variable with mean $0$ and variance $\sigma_\ri^2$,
\[
\PAR{\frac{1}{\sqrt t}\int_{0}^{t}\left(I(s) -\dE[I(s)]\right)\rd s}\overset{\rm d}{\underset{t\to \infty}{\longrightarrow}}Z^\ri.
\]
\end{assumption}
Let $\dZZ^{\ri}$ be a Gaussian random variable with mean zero and variance given by
\[
\var(\dZZ^{\ri})=\rho\,\var(X(1))+(N-\rho)\,\var(Y(1))+(\dE[X(1)]-\dE[Y(1)])^2\sigma_{\ri}^2,
\] 
where $\rho$ is defined in Proposition~\ref{prop:LLN for W}.

\begin{proposition}[{\sc clt} for $W$]\label{prop:CTL for W}
    Under Assumptions~\ref{A:hyp LLN for W} and~\ref{A:hyp TCL for W}, if  $X(1)$ and $Y(1)$ have finite second moments, then
\begin{equation}
\frac{1}{\sqrt{t}}\PAR{W(t)-\dE[W(t)]}\overset{\rm d}{\underset{t\to \infty}{\longrightarrow}}\dZZ^{\rm i}.
\end{equation}
\end{proposition}

\begin{proof}
Whereas the {\sc slln} for $W$ followed directly from its counterpart for $V$, one cannot follow the same route when proving the {\sc clt}, because the random processes $W^j$ are dependent. We therefore present a direct proof, similar to the proof of Proposition~\ref{prop:CLT for V} for the process $V$.

By Lemma~\ref{lem:characteristic function W}, we have, for any fixed $\alpha\in\dR$,
\begin{align*}
\widetilde m_\alpha(t):=\dE\SBRA{\re^{\frac{i\alpha}{\sqrt{t}}\PAR{W(t)-\dE[W(t)]}}}&=\re^{\varphi_Y({{\alpha}/{\sqrt{t}}})Nt}\,\dE\SBRA{\exp\left({ \PAR{\varphi_X\PAR{\frac{\alpha}{\sqrt{t}}}-\varphi_Y\PAR{\frac{\alpha}{\sqrt{t}}}}\int_0^tI(s)\,\rd s}\right)}\\
& \times\exp\left({-i\frac{\alpha}{\sqrt{t}}\PAR{\PAR{\dE[X(1)]-\dE[Y(1)]}\int_0^t\dE[I(s)]\,\rd s+\dE[Y(1)]Nt}}\right).
\end{align*}
Using standard Taylor expansions, we have when $t$ goes to $\infty$,
\begin{align*}
\varphi_X\PAR{\frac{\alpha}{\sqrt{t}} }&=i\frac{\alpha}{\sqrt{t}}\dE[X(1)]-\frac{\alpha^2}{2t} \var(X(1))+\widetilde \varepsilon_X(t),\\
\varphi_Y\PAR{\frac{\alpha}{\sqrt{t}} }&=i\frac{\alpha}{\sqrt{t}} \dE[Y(1)]-\frac{\alpha^2}{2t}\var(Y(1))+\widetilde \varepsilon_Y(t),
\end{align*}
with $\widetilde \varepsilon_X, \widetilde \varepsilon_Y$ bounded functions such that $\lim_{t\to \infty}t\widetilde \varepsilon_X(t)=0$ and $\lim_{t\to \infty}t\widetilde \varepsilon_Y(t)=0$.
Consequently, with $\widetilde \varepsilon(t):=\widetilde \varepsilon_X(t)-\widetilde \varepsilon_Y(t)$ and $\bar\rho:=\dE[X(1)]-\dE[Y(1)]$,
\begin{align*}
\widetilde m_\alpha(t)
&=\exp\left({-\frac{\alpha^2}{2}N\var(Y(1))+t\widetilde \varepsilon_Y(t)}\right)\,\dE\Bigg[\exp\left({i\frac{\alpha}{\sqrt{t}}\,\bar\rho\int_0^t\PAR{I(s)-\dE[I(s)]}\,\rd s}\right)\\
&\quad\quad\quad\hspace{1.3cm}\exp\left({-\frac{\alpha^2}{2t} \PAR{\var(X(1))-\var(Y(1))}\int_0^tI(s)\,\rd s}\right)\,\exp\left({\widetilde \varepsilon(t)\int_0^tI(s)\,\rd s}\right)\Bigg].
\end{align*}
Due to the assumptions that we imposed, we have
\begin{align*}
&\widetilde H(t):=\frac{1}{\sqrt t}\int_{0}^{t}\left(I(s) -\dE[I(s)]\right)\,\rd s\overset{\rm d}{\underset{t\to \infty}{\longrightarrow}}Z^{\ri},\\
&\widetilde G(t):=\frac{1}
{t}\PAR{\var(X(1))-\var(Y(1))}\int_0^tI(s)\,\rd s\overset{\rm {\small a.s.}}{\underset{t\to \infty}{\longrightarrow}}
\widetilde\eta,
\end{align*}
with $\widetilde \eta:=\PAR{\var(X(1))-\var(Y(1))}\,\rho$.
Consequently, there is a constant $C>0$ such that
\begin{align*}
\ABS{\widetilde m_\alpha(t)-\re^{-\frac{\alpha^2}{2}\var({\dZZ^\ri})}}
&=\re^{-\frac{\alpha^2}{2}N\,\var(Y(1))}\ABS{\dE\SBRA{\re^{-t\widetilde \varepsilon_Y(t)}\re^{i\alpha\,\bar\rho \,\widetilde H(t)}\re^{-\frac{\alpha^2}{2}\widetilde G(t)+\widetilde \varepsilon(t)}}-\re^{-\frac{\alpha^2}{2}\,(\bar\rho^2\,\sigma_{\ri}^2+\widetilde \eta)}}\\
&\leq \re^{-\frac{\alpha^2}{2}N\,\var(Y(1))}\Bigg(\ABS{\re^{-t\widetilde \varepsilon_Y(t)}-1}+C\,\dE\SBRA{\ABS{\re^{-\frac{\alpha^2}{2}\widetilde G(t)+\widetilde \varepsilon(t)}-\re^{-\frac{\alpha^2}{2}\widetilde \eta}}}\\
&\hskip 3cm +C\,\re^{-\frac{\alpha^2}{2}\widetilde \eta}\,\ABS{\dE\SBRA{\re^{i\alpha\,\bar\rho\,\widetilde H(t)}}-\re^{-\frac{\alpha^2}{2}\,\bar\rho^2\,\sigma_\ri^2}}\Bigg).
\end{align*}
All the terms on the right-hand side goes to $0$ as $t\to \infty$, by dominated convergence for the second term and by Assumption~\ref{A:hyp TCL for W} for the last term. This proves the claim. 
\end{proof}

Note that the expression for $\var(\dZZ^{\rm i}  )$ is in accordance with the limit of $\var(W(t))/t$ as $t\to\infty$, with $\var(W(t))$ as given in Lemma \ref{lem:esp-variance}.

\section{Ruin probability}\label{sec:ruin probability}

A primary goal of this paper is to obtain information on the ruin probability, defined as the probability that the surplus process, starting from an initial reserve $u$, becomes negative. In this section, we focus on the collective model. This choice is motivated by the fact that, as will be illustrated in the applications presented in Section~\ref{sec:application individual model}, there already exists a substantial body of literature on ruin probabilities and their decay rates in settings where $I$ is an irreducible continuous-time Markov process. In that case, the model falls within the framework of Markov additive processes (MAPs), for which ruin probabilities and their tail behavior have been extensively studied.

\medskip
We consider the risk process $V$ defined by \eqref{eq:def-V}, which is such that $V(0)=0$. Hence, the ruin probability, given that the initial level is $u\geq 0$, is defined by
\begin{equation}\label{eq:def ruin proba V}
p(u):=\dP\PAR{u+\inf_{t\geq 0}V(t)<0},
\end{equation}
and the ruin time is given by 
\begin{equation}\label{eq:def ruin time V}
\tau(u):=\inf\{t\geq0:u+V(t)<0\}.
\end{equation}
To avoid trivialities, in what follows, we assume that $X$ and $Y$ are not both subordinators.

From Corollary~\ref{coro:LLN for V}, we obtain the following result, which is analogous to its classical version; see, e.g., \cite[Corollary~IV.1.4]{Asmussen-Albrecher10}.

\begin{lemma}\label{lem:ruin-condition}
Under Assumption~\ref{A:hyp LLN for V}, if
    $X(1)$ and $Y(1)$ have  finite second moments, then
    \begin{enumerate}
    \item If $\rho\dE[X(1)]+(1-\rho)\dE[Y(1)]<0$, then $V(t)\underset{t\to \infty}{\longrightarrow}-\infty$ a.s. and $p(u)=1$.
    \item If $\rho\dE[X(1)]+(1-\rho)\dE[Y(1)]>0$, then $V(t)\underset{t\to \infty}{\longrightarrow}+\infty$ a.s. and  $p(u)<1$.
    \end{enumerate}
\end{lemma}

\medskip
From Lemma~\ref{lem:characteristic function V}, for $t\geq 0$ and $\alpha\in\dR$ such that the expectation is well defined, we have 
\begin{equation}\label{eq:laplace function V}
\dE\SBRA{\re^{-\alpha V(t)}}=\dE\SBRA{\re^{k_{\alpha}(t)}},
\end{equation}
where the \textit{random} Laplace exponent is given by 
\begin{equation}\label{eq:k alpha}
\begin{aligned}
k_\alpha(t)
={}& \int_0^t \varphi_X(i \alpha F(s))\,\rd s
    + \int_0^t \varphi_Y(i \alpha (1-F(s)))\,\rd s \\[0.3em]
={}& - \alpha \left(
    \rc_X \int_0^t F(s)\,\rd s
    + \rc_Y \int_0^t (1-F(s))\,\rd s
\right) \\[0.3em]
&+ \frac{\alpha^2}{2} \left(
    \sigma_X^2 \int_0^t F^2(s)\,\rd s
    + \sigma_Y^2 \int_0^t (1-F(s))^2\,\rd s
\right) \\[0.3em]
&+ \int_0^t \int_{\dR \setminus \BRA{0}}
\Big(
    e^{-\alpha F(s)x}
    - 1
    + \alpha F(s)x \ind_{\ABS{x}<1}
\Big)\,\nu_X(\rd x)\,\rd s \\[0.3em]
&+ \int_0^t \int_{\dR \setminus \BRA{0}}
\Big(
    e^{-\alpha (1-F(s))x}
    - 1
    + \alpha (1-F(s))x \ind_{\ABS{x}<1}
\Big)\,\nu_Y(\rd x)\,\rd s .
\end{aligned}
\end{equation}

Note that $\re^{k_\alpha(t)}=\dE\SBRA{\re^{-\alpha V(t)}\,\vert\, \cF^F}$.

We introduce the following exponential martingale related to the process $V$ (see Appendix~\ref{Appendix:martingale} for a proof).

\begin{lemma}\label{lem: Martingale}
Assume there exists $\overline{\alpha}>0$ such that both $\dE[\re^{-\overline{\alpha} X(1)}]$ and $\dE[\re^{-\overline{\alpha} Y(1)}]$ are finite. Then, for any $\alpha \in [0,\overline{\alpha}]$, the process $M_\alpha$ defined by
\[
M_\alpha(t) := \exp\PAR{- \alpha V(t) - k_\alpha(t)}
\] 
is a càdlàg $\PAR{\cF_t}_{t\geq 0}$-martingale, with $M_\alpha(0)=1$. It is also an $\PAR{\cF_t \vee \cF^F}_{t\geq 0}$-martingale.

Moreover, if $\tau$ is a bounded $\PAR{\cF_t \vee \cF^F}_{t\geq 0}$-stopping time, then
\begin{equation}\label{eq:conditional-Doob}
\dE\big[M_\alpha(\tau)\,\vert\, \cF^F\big] = 1, \quad \text{almost surely.}
\end{equation}
\end{lemma}

In the remainder of this section, we first derive rough bounds for the ruin probability $p(u)$, and then establish an identity satisfied by $p(u)$.  As before, the models considered are defined as specified in the 
\emph{Stochastic objects} paragraph of the introduction.

{In addition, we introduce a `Cram\'er-type assumption': there exists $\overline{\alpha}>0$ such that}
\begin{equation}\label{cram}
\dE[\re^{-\overline{\alpha} X(1)}]<\infty,\quad \dE[\re^{-\overline{\alpha} Y(1)}]<\infty.
\end{equation}

\subsection{Rough bounds on the ruin probability}

{The first result of this section establishes that, under specific conditions including the Cram\'er-type assumption given in~\eqref{cram}, one can derive an explicit exponential upper bound for $p(u)$.}

In this direction, let us introduce $f_{\min}=\operatorname*{ess\inf}_{t\geq 0}F(t)$ and $f_{\max}=\operatorname*{ess\sup}_{t\geq 0}F(t)$, which belong to $[0,1]$, and $C_\alpha$ the function defined on $[0,1]$ by
\begin{align}\label{eq:definition of C}
C_\alpha(z) 
:=& -(\rc_X z + \rc_Y (1-z)) + \frac{\alpha}{2} \PAR{\sigma_X^2 z^2 + \sigma_Y^2 (1-z)^2} \notag\\
& + \frac{1}{\alpha} \int_{\dR \setminus \BRA{0}} \left(e^{-\alpha z x} - 1 + \alpha z x \ind_{\ABS{x}<1}\right) \nu_X(\rd x) \notag\\
& + \frac{1}{\alpha} \int_{\dR \setminus \BRA{0}} \left(e^{-\alpha (1-z) x} - 1 + \alpha (1-z) x \ind_{\ABS{x}<1}\right) \nu_Y(\rd x).
\end{align}

\begin{proposition}\label{prop:estimation ruin probability for the collective model}
Assume that $X(1)$ and $Y(1)$ have  finite first moments, and assume:
\begin{itemize}
    \item[{\rm (1)}] {the Cram\'er condition \eqref{cram} applies;} 
\item[{\rm (2)}] 
$\min(f_{\min}\dE[X(1)]+(1-f_{\min})\dE[Y(1)],f_{\max}\dE[X(1)]+(1-f_{\max})\dE[Y(1)])>0.$
\end{itemize}
Then, for any $u>0$, 
\begin{equation}
p(u)\leq \re^{-\alpha_0 u},
\end{equation}
where $\alpha_0=\min\BRA{\widetilde\alpha,\overline\alpha}$, with $\widetilde\alpha:=\sup\BRA{\alpha>0: C_\alpha(f_{\min})<0\text{ and }C_\alpha(f_{\max})<0}$.
\end{proposition}

For spectrally positive Lévy processes, including arithmetic Brownian motion, we observe that $\alpha_0 = \widetilde\alpha$, since Lemma~\ref{lem: Martingale} holds for any $\alpha>0$. Moreover, as long as $\dE[X(1)]$ and $\dE[Y(1)]$ are positive, the second condition in Proposition~\ref{prop:estimation ruin probability for the collective model} is automatically satisfied, ensuring that the assumptions required for the exponential bound are met.

\begin{proof}
 Let $u,\alpha,t>0$ with $\alpha\leq \overline{\alpha}$. The result being obvious if $p(u)= 0$, we assume that the ruin probability  is positive. As $M_\alpha=(\re^{-\alpha V(t)-k_\alpha(t)})$ is a $\PAR{\cF_t}_{t\geq 0}$-martingale (Lemma~\ref{lem: Martingale}), from Doob's optional stopping theorem, we have
 \begin{align*}
 1&=\dE\left[\re^{-\alpha V(t\wedge\tau(u))-k_\alpha(t\wedge\tau(u))}\right]
 \geq \dE\left[\re^{-\alpha V(\tau(u))-k_\alpha(\tau(u))}\ind_\BRA{\tau(u)\leq t}\right].
 \end{align*}
  By the monotone convergence theorem, it follows that 
 \begin{align}
 1&\geq\dE\left[\re^{-\alpha V_{\tau(u)}-k_{\tau(u)}(\alpha)}\ind_\BRA{\tau(u)<\infty}\right]=\dE\left[\re^{-\alpha V_{\tau(u)}-k_{\alpha}(\tau(u))}\big|\tau(u)<\infty\right]p(u)\nonumber\\
 &\geq\re^{\alpha u}\dE\left[\re^{-k_{\alpha}(\tau(u))}\big|\tau(u)<\infty\right]p(u),\label{eq:estimation probability of ruin}
 \end{align}
 since $V_{\tau(u)}\leq-u$ by the definition~\eqref{eq:def ruin time V} of the ruin time $\tau(u)$.

To conclude, it suffices to show that $k_\alpha(t)\leq 0$ for any $0<\alpha\leq\alpha_0$, for a suitably chosen $\alpha_0$. 
We note that $k_\alpha(0)=0$, and the derivative of $k_\alpha$ is given by
\begin{align*}
k_\alpha'(t) = \alpha\, C_\alpha(F(t)) .
\end{align*}
By convexity of $z \mapsto C_\alpha(z)$, see \eqref{eq:definition of C}, it follows that, for any $t \geq 0$,
\[
k_\alpha'(t) \leq \alpha \max\big(C_\alpha(f_{\min}), C_\alpha(f_{\max})\big).
\]
We now seek an $\alpha_0>0$ such that, for all $\alpha \leq \alpha_0$,
\[
\max\big(C_\alpha(f_{\min}), C_\alpha(f_{\max})\big) \leq 0.
\]
To this end, we compute the limit when $\alpha$ goes to $0$ of the continuous function $\alpha\mapsto C_\alpha(z)$  for a fixed $z\in[0,1]$. Set $z\in\dR$, and note that the functions $\alpha\mapsto ({\re^{-\alpha z}-1})/{\alpha}$ and $\alpha\mapsto ({\re^{-\alpha z}-1+\alpha z})/{\alpha}$  are increasing on $(0,\overline\alpha]$, with respective limits $-z$ and $0$ when $\alpha\to 0$. As the processes $X$ and $Y$ have a finite first-order moment and a finite Laplace transform of order $\overline{\alpha}$, we deduce by the dominated convergence theorem that 
\begin{align*}
C_0(z):=\lim_{\alpha\to 0}C_\alpha(z)&=-(\rc_Xz+\rc_Y(1-z))-z\int_{\dR\setminus[-1,1]}x\nu_X(\rd x)-(1-z)\int_{\dR\setminus[-1,1]}x\nu_Y(\rd x)\\
&=-(z\dE[X(1)]+(1-z)\dE[Y(1)]).
\end{align*} 
By assumption, we have $C_0(f_{\min})<0$ and $C_0(f_{\max})<0$. Introducing 
\[
\alpha_0:=\sup\BRA{\alpha>0: C_\alpha(f_{\min})<0\text{ and }C_\alpha(f_{\max})<0},
\]
we obtain $k_\alpha(t)\leq 0$ for any $\alpha\leq \alpha_0$ and any $t\geq0$. We conclude by Inequality \eqref{eq:estimation probability of ruin}.
\end{proof}

We easily deduce the following corollary when $X$ and $Y$ are spectrally positive Lévy risk processes, including arithmetic Brownian motions.
Define 
\[\alpha_0=2\min\BRA{\frac{f_{\min}\dE[X(1)]+(1-f_{\min})\dE[Y(1)]}{f_{\min}^2\widetilde\sigma_X^2+(1-f_{\min})^2\widetilde\sigma_Y^2},\frac{f_{\max}\dE[X(1)]+(1-f_{\max})\dE[Y(1)]}{f_{\max}^2\widetilde\sigma_X^2+(1-f_{\max})^2\widetilde\sigma_Y^2}},\] 
in which $\widetilde\sigma_X^2 := \sigma_X^2+\int x^2\nu_X(\rd x)$ and $\widetilde\sigma_Y^2 := \sigma_Y^2+\int x^2\nu_Y(\rd x)$.

\begin{corollary}\label{coro:estimate ruin spectrally positive}
Assume $X$ and $Y$ are (independent) spectrally positive Lévy risk processes with finite first moments, and assume that
\[
\min\BRA{f_{\min}\dE[X(1)]+(1-f_{\min})\dE[Y(1)],f_{\max}\dE[X(1)]+(1-f_{\max})\dE[Y(1)]}>0.\]
Then, for any $u>0$, 
\begin{equation}
p(u)\leq \re^{-\alpha_0 u}.
\end{equation}
\end{corollary}

Before proving this corollary, let us discuss a well-known result for arithmetic Brownian motion.

\begin{remark}
\em 
Assume $X$ and $Y$ are arithmetic Brownian motions and assume that the process $F \equiv \rho \in [0,1]$ is constant. 
In this case, by Corollary~\ref{coro:estimate ruin spectrally positive}, we have
\[
\alpha_0 = 2\frac{\rho \rc_X + (1-\rho) \rc_Y}{\rho^2 \sigma_X^2 + (1-\rho)^2 \sigma_Y^2}.
\]
Under these conditions, the risk process $V$ itself is an arithmetic Brownian motion with drift
\[
\mu_V: = \rho \rc_X + (1-\rho) \rc_Y
\]
and diffusion coefficient
\[
\sigma_V := \sqrt{\rho^2 \sigma_X^2 + (1-\rho)^2 \sigma_Y^2}.
\]
It is therefore well known (see, e.g., \cite[Corollary~II.2.4]{Asmussen-Albrecher10}) that the ruin probability is
\[
p(u) = \exp \PAR{-2 \frac{\rho \rc_X + (1-\rho) \rc_Y}{\rho^2 \sigma_X^2 + (1-\rho)^2 \sigma_Y^2} \, u}.
\]
Consequently, the bound for the ruin probability given in Corollary~\ref{coro:estimate ruin spectrally positive} is exact when $F$ is constant and $X, Y$ are Brownian risk processes. 

However, it should be noted that, when $F$ is not constant, the rates provided in Proposition~\ref{prop:estimation ruin probability for the collective model} and in Corollary~\ref{coro:estimate ruin spectrally positive} may no longer be optimal, as highlighted by our numerical experiments in Section~\ref{sec:application collective model}.\hfill $\Diamond$
\end{remark}

\begin{proof}[Proof of Corollary~\ref{coro:estimate ruin spectrally positive}]
Recall that
\[
e^{-z} \leq 1 - z + \frac{z^2}{2}, \quad \text{for any } z \geq 0.
\]
Then, for spectrally positive Lévy risk processes, by the definition~\eqref{eq:definition of C} of $C_\alpha$, we have, for $z \in [0,1]$ and $\alpha > 0$,
\begin{align*}
C_\alpha(z) 
&\leq -\big(\dE[X(1)]\, z + \dE[Y(1)]\, (1-z)\big) 
+ \frac{\alpha}{2} \PAR{\widetilde\sigma_X^2 z^2 + \widetilde\sigma_Y^2 (1-z)^2},
\end{align*}
where 
$\widetilde\sigma_X^2$ and $\widetilde\sigma_Y^2$ are as defined above.
It is then immediate that 
$
C_\alpha(f_{\min}) \leq 0$ and $C_\alpha(f_{\max}) \leq 0$
for any $\alpha \leq \alpha_0$, with $\alpha_0$ as given in Corollary~\ref{coro:estimate ruin spectrally positive}. 
The conclusion now follows directly from Proposition~\ref{prop:estimation ruin probability for the collective model}.
\end{proof}

\subsection{An identity for the ruin probability} 

In this subsection, we present a theoretical identity characterizing the ruin probability. 

\begin{thm}\label{thm: exact ruin  for the collective model}
Assume:
\begin{enumerate}
\item[{\em (1)}] {the Cram\'er condition \eqref{cram} applies};
\item[{\em (2)}] there exists a deterministic $\ell_\alpha$ such that 
\[
\ell_\alpha = \lim_{t \to \infty} \frac{k_\alpha(t)}{t} \quad \text{almost surely, for all } \alpha \in [0, \overline{\alpha}];
\]
\item[{\em (3)}] $\alpha_*:=\sup\BRA{\alpha>0 :\ell_{\alpha}=0}\in (0, \overline{\alpha})$.
\end{enumerate}

Then, the ruin probability $p(u)$ is positive and satisfies, for any $\alpha \in (\alpha_*, \overline{\alpha}]$,
\begin{equation*}
p(u) = \frac{1}{\dE\big[\re^{-\alpha V(\tau(u)) - k_{\alpha}(\tau(u))} \,\big|\, \tau(u) < \infty\big]}.
\end{equation*}
\end{thm}

\begin{proof}
First, note that, for any $t$, we have that $k_0(t)=0$ and $\alpha\mapsto k_\alpha(t)$ is convex. Consequently, $\ell_0=0$ and $\alpha\mapsto \ell_\alpha$ is convex. In addition, we have $\ell_\alpha\leq 0$ when $\alpha\leq \alpha_*$, and $\ell_\alpha> 0$ when $\alpha> \alpha_*$. We deduce that, if $\alpha>\alpha_*$, then $k_\alpha(t)>0$ for $t$ sufficiently large.

\medskip

First, assume that $F$ is deterministic, and therefore $k_\alpha$ is deterministic. Let $\alpha, t > 0$, with $\alpha \leq \overline{\alpha}$ and $\alpha > \alpha_*$. 
Since $M_\alpha$ is a càdlàg $\PAR{\cF_t}_{t \geq 0}$-martingale by Lemma~\ref{lem: Martingale}, Doob's optional stopping theorem yields
\begin{align}\label{eq: Doob thm for Martingales}
1 &= \dE\left[\re^{-\alpha V(t \wedge \tau(u)) - k_\alpha(t \wedge \tau(u))}\right] \nonumber\\
&= \dE\left[\re^{-\alpha V(\tau(u)) - k_\alpha(\tau(u))} \ind_\BRA{\tau(u) \leq t}\right]
+ \dE\left[\re^{-\alpha V(t) - k_\alpha(t)} \ind_\BRA{\tau(u) > t}\right].
\end{align}

By assumption, $\lim_{t \to \infty} k_\alpha(t)/t = \ell_\alpha > 0$, so there exists $T > 0$ such that for all $t > T$ it holds that
$k_\alpha(t) \geq \tfrac{1}{2}{\ell_\alpha}\,t.$
Consequently, for $t > T$,
\[
\dE\Big[\re^{-\alpha V(t) - k_\alpha(t)} \ind_\BRA{\tau(u) > t}\Big] \leq \re^{\alpha u} \re^{-\frac{\ell_\alpha}{2} t},
\]
and therefore,
\[
\lim_{t \to \infty} \dE\Big[\re^{-\alpha V(t) - k_\alpha(t)} \ind_\BRA{\tau(u) > t}\Big] = 0.
\]

On the other hand, by the monotone convergence theorem,
\[
\lim_{t \to \infty} \dE\Big[\re^{-\alpha V(\tau(u)) - k_\alpha(\tau(u))} \ind_\BRA{\tau(u) \leq t}\Big] 
= \dE\Big[\re^{-\alpha V(\tau(u)) - k_\alpha(\tau(u))} \ind_\BRA{\tau(u) < \infty}\Big].
\]

Taking the limit $t \to \infty$ in Equation~\eqref{eq: Doob thm for Martingales}, we obtain
\begin{equation}\label{eq:limit in Doob}
1 = \dE\Big[\re^{-\alpha V(\tau(u)) - k_\alpha(\tau(u))} \ind_\BRA{\tau(u) < \infty}\Big].
\end{equation}
We now show by contradiction that $p(u) > 0$. Suppose $\tau(u) = \infty$ almost surely, so that $V(t) \geq -u$ for all $t \geq 0$. 
Using again the martingale $M_\alpha = \PAR{\re^{-\alpha V(t) - k_\alpha(t)}}_{t \geq 0}$, we have
\[
1 = \dE\left[\re^{-\alpha V(t) - k_\alpha(t)}\right] \leq \re^{\alpha u} \re^{-k_\alpha(t)}.
\]

For $t > T$, $k_\alpha(t) \geq \tfrac{1}{2}{\ell_\alpha}\, t$ with $\ell_\alpha > 0$, and taking $t \to \infty$ leads to the contradiction $1 \leq 0$. Hence, $p(u) > 0$, and the conclusion follows by dividing equality~\eqref{eq:limit in Doob} by $p(u)$.
\medskip

Second, assume that $F$ is a random process. In this case, the time $T$ introduced in the deterministic setting becomes a $\cF^F$-measurable random variable, finite almost surely. Using~\eqref{eq:conditional-Doob} from Lemma~\ref{lem: Martingale} with the bounded stopping time $t \wedge \tau(u)$, we have
\begin{align*}
1 &= \dE\Big[\re^{-\alpha V(\tau(u) \wedge t) - k_\alpha(\tau(u) \wedge t)} \,\big|\, \cF^F\Big] \\
&= \dE\Big[\re^{-\alpha V(\tau(u)) - k_\alpha(\tau(u))} \ind_\BRA{\tau(u) \leq t} \,\big|\, \cF^F\Big]
+ \dE\Big[\re^{-\alpha V(t) - k_\alpha(t)} \ind_\BRA{\tau(u) > t} \,\big|\, \cF^F\Big].
\end{align*}
The event $\BRA{t>T}$ belongs to $\cF^F$. Therefore, the same arguments as in the deterministic case apply on this event. Letting  $t \to \infty$, we deduce
\[
1 = \dE\Big[\re^{-\alpha V(\tau(u)) - k_\alpha(\tau(u))} \ind_\BRA{\tau(u) < +\infty} \,\big|\, \cF^F\Big], \quad \text{almost surely}.
\]
Arguing exactly as in the deterministic case and applying conditional expectations with respect to $\cF^F$, we also prove that $p(u) > 0$. Taking expectations and dividing both sides by $p(u)$ then yields the result.
\end{proof}
\begin{remark} \em
Since $V(\tau(u)) \leq -u$ on the event $\BRA{\tau(u) < \infty}$, under the assumptions of Theorem~\ref{thm: exact ruin  for the collective model} we immediately obtain that for $\alpha > \alpha_*$,
\[
p(u) \leq \frac{\re^{-\alpha u}}{\dP\big(\alpha(\tau(u)) \geq \alpha \,\big|\, \tau(u) < \infty\big)},
\]
where $\alpha(t) = \inf \BRA{\alpha > 0 : k_\alpha(t) > 0}$ denotes the maximal root of the mapping $\alpha \mapsto k_\alpha(t)$.  

Applying the monotone convergence theorem as $\alpha \downarrow \alpha_*$, we then have
\[
p(u) \leq \frac{\re^{-\alpha_* u}}{\dP\big(\alpha(\tau(u)) \geq \alpha_* \,\big|\, \tau(u) < \infty\big)}.
\]

Note that when $F$ is constant, $\alpha(t) = \alpha_*$ does not depend on $t$, and we recover the classical Lundberg inequality:
\begin{equation}\label{eq:lundberg inequality}
p(u) \leq \re^{-\alpha_* u}.
\end{equation}

However, as highlighted in Section~\ref{sec:application collective model}, the Lundberg inequality \eqref{eq:lundberg inequality}, with $\alpha_*$ defined as in Theorem~\ref{thm: exact ruin  for the collective model}, does not, in general, hold for a heterogeneous time-dependent population. This is particularly striking in view of the well-known robustness of exponential bounds derived under Cramér-type conditions. Indeed, in classical risk models, such bounds apply under remarkably broad assumptions and exhibit a high degree of universality. The present setting thus illustrates a genuine limitation of these classical exponential estimates.
\hfill $\Diamond$
\end{remark}

\begin{remark} \em
When $X$ and $Y$ are spectrally negative Lévy processes, and when $F$ is ergodic or convergent, several assumptions of Theorem~\ref{thm: exact ruin  for the collective model} are automatically satisfied.  
Specifically, when $F$ is an ergodic Markov process (or when $F$ converges almost surely to a deterministic value $\rho$), both the limit $\rho$ from Assumption~\ref{A:hyp LLN for V} and the quantity $\ell_\alpha$ from the assumptions of Theorem~\ref{thm: exact ruin  for the collective model} exist. In this case, $\ell_\alpha$ is given by
\begin{align}\label{eq:l alpha}
\ell_\alpha &= 
-\alpha \PAR{\rc_X \rho + \rc_Y (1-\rho)}
+ \frac{\alpha^2}{2} \PAR{\sigma_X^2 \dE[Z^2] + \sigma_Y^2 \dE[(1-Z)^2]} \notag\\
&\quad + \int_{\dR \setminus \BRA{0}} 
\Bigl( \dE\SBRA{ e^{-\alpha x Z} } - 1 + \alpha \, \dE[Z] \, x \, \ind_{\ABS{x}<1} \Bigr) \, \nu_X(\rd x) \notag\\
&\quad + \int_{\dR \setminus \BRA{0}} 
\Bigl( \dE\SBRA{ e^{-\alpha (1-Z) x} } - 1 + \alpha \, (1-\dE[Z]) \, x \, \ind_{\ABS{x}<1} \Bigr) \, \nu_Y(\rd x).
\end{align}
where $Z$ is a random variable with distribution equal to the invariant probability measure of $F$ (or $Z=\rho$ in the deterministic case). 
Moreover, under the net profit condition
\[
\rho\,\dE[X(1)] + (1-\rho)\,\dE[Y(1)] > 0,
\]
Expression~\eqref{eq:l alpha} guarantees the existence of a unique positive root $\alpha_*$ such that $\ell_{\alpha_*} = 0$. Consequently, the result of Theorem~\ref{thm: exact ruin  for the collective model} holds for any $\alpha \in (\alpha_*, \max\BRA{\alpha_*, \overline\alpha}]$.
\hfill $\Diamond$
\end{remark}

We now describe an interesting relationship between the positive roots of $k_\alpha(t)$ and of $\ell_\alpha$.

\begin{proposition}\label{prop:convergence of the roots}
Assume:
\begin{enumerate}
\item[{\em (1)}] {the Cram\'er condition \eqref{cram} applies};
\item[{\em (2)}] for $\alpha\geq 0$, the limit $\ell_\alpha := \lim_{t\to \infty} k_\alpha(t)/t$ exists almost surely;
\item[{\em (3)}] there is a unique $\alpha_*\in(0,\overline{\alpha})$ such that $\ell_{\alpha_*}=0$.
\end{enumerate}
Let $\alpha(t)=\sup\BRA{\alpha\geq 0: k_\alpha(t)=0}$ denote the largest root of $\alpha\mapsto k_\alpha(t)$. Then
\[
\lim_{t\to \infty} \alpha(t) = \alpha_* \quad \text{almost surely.}\]
In addition, we have $\alpha_0\leq \alpha_*$, where $\alpha_0$ is the rough decay rate introduced in Proposition~\ref{prop:estimation ruin probability for the collective model}.
\end{proposition}

\begin{proof}  
Both $\alpha\mapsto k_\alpha(t)$ and $\alpha\mapsto \ell_\alpha$ are convex. Since $\ell_\alpha>0$ for $\alpha>\alpha_*$, it follows that $k_\alpha(t)>0$ for large $t$ when $\alpha>\alpha_*$.  
Let $\varepsilon>0$. For sufficiently large $t$, $\alpha(t)\in [0,\alpha_*+\varepsilon]$, so there exists a random subsequence $(t_n)_{n\ge 1}$ with $t_n\to\infty$ and a random variable $\widetilde\alpha_*\in [0,\alpha_*+\varepsilon]$ such that $\alpha(t_n)\to \widetilde\alpha_*$ almost surely. 

We now show that $\widetilde\alpha_*=\alpha_*$.  
By definition, $k_{\alpha(t_n)}(t_n)/t_n=0$ a.s.\ for all $n$. Taking $n\to\infty$ and using the continuity of $(\alpha,t)\mapsto k_\alpha(t)/t$, we obtain $\ell_{\widetilde\alpha_*}=0$. By uniqueness of the positive root, $\widetilde\alpha_*\in\{0,\alpha_*\}$.  

Assume, for contradiction, that $\widetilde\alpha_*=0$ on a set of positive probability. Then for large $n$, $\alpha(t_n)\leq \alpha_*/2$, which implies $k_\alpha(t_n)>0$ for all $\alpha>\alpha_*/2$. Taking the limit as $n\to\infty$ gives $\ell_{\alpha_*/2}\ge 0$, contradicting the uniqueness of $\alpha_*$. Hence $\widetilde\alpha_*=\alpha_*$ almost surely.  
Since every convergent subsequence tends to the same limit, we conclude that $\alpha(t)\to\alpha_*$ almost surely, as $t\to\infty$.

Finally, we prove that  $\alpha_0\leq \alpha_*$ by contradiction. The value $\alpha_0$ was chosen in the proof of Proposition~\ref{prop:estimation ruin probability for the collective model} so that for all $\alpha\leq \alpha_0$ and $t\geq 0$, $k_\alpha(t)\leq 0$, which implies that $\alpha(t)\geq \alpha_0$. If $\alpha_*<\alpha_0$, the contradiction follows from the convergence of $\alpha(t)$ to $\alpha_*$.
\end{proof}

We conclude this section with the specific case of Brownian risk processes.
We consider the case that $X$ and $Y$ are two independent arithmetic Brownian motions with respective drift $\rc_X,\rc_Y$ and diffusion coefficients $\sigma_X,\sigma_Y>0$.
We have
\[
       k_\alpha(t)=-\alpha\PAR{\rc_X\int_0^tF(s)\rd s+\rc_Y\int_0^t(1-F(s))\rd s}
       +\frac{\alpha^2}{2}\PAR{\sigma_X^2\int_0^tF^2(s)\rd s+\sigma_Y^2\int_0^t(1-F(s))^2\rd s}.
\]

\begin{proposition}\label{prop: exact ruin brownian motion for collective model}
Let $X$ and $Y$ be two independent arithmetic Brownian motions. We assume the existence of $\rho$ and $\eta$ defined in Assumptions~\ref{A:hyp LLN for V} and \ref{A:hyp CLT for V}. Set $\alpha_*:=2\frac{\rho \rc_X+(1-\rho)\rc_Y}{\sigma_X^2\eta^2+\sigma_Y^2\left(\eta^2+1-2\rho\right)}$. 

Under the net profit condition $\rho\rc_X+(1-\rho)\rc_Y>0$, for any $\alpha>\alpha_*$,
\begin{equation}
p(u)=\frac{\re^{-\alpha u}}{\dE\left[\re^{-k_{\alpha}(\tau(u))}\,\big|\,\tau(u)<\infty\right]}.
\end{equation}
Moreover, when $F$ is deterministic, the decay rate of $p(u)$ is $\alpha_*$:
\[
\lim_{u\to \infty}\frac{1}{u}\log p(u)= -\alpha_*.
\]
\end{proposition}

\begin{proof} 
 Let us note that $X(1)$ and $Y(1)$ possess finite Laplace transforms. The limit of the sequence $k_\alpha(t)/t$ is
\[
\ell_\alpha = -\alpha\big(\rho \rc_X + (1-\rho)\rc_Y\big) + \frac{\alpha^2}{2}\big(\sigma_X^2 \eta^2 + \sigma_Y^2 (\eta^2 + 1 - 2\rho)\big).
\]
Under the net profit condition, it follows immediately that $\alpha_*$ has the claimed form.  

Moreover, since the process $V$ is continuous, we have $V(\tau(u))=-u$. The first result then follows directly from Theorem~\ref{thm: exact ruin  for the collective model}.

We now deduce the decay rate of $p(u)$ in the case of deterministic $F$, following \cite{duffield-oconnell95}. Since $F$ is deterministic, we have
\[
\ell_\alpha = \lim_{t\to\infty} k_\alpha(t)/t = \lim_{t\to\infty} \frac{1}{t} \log \dE\SBRA{\exp(-\alpha V_t)}.
\]
It is straightforward to verify that the process $(-V_t)_{t\geq 0}$ satisfies \cite[Assumption 2.1]{duffield-oconnell95} with $v_t = a_t = t$ and $g(u) = 1/u$.  
For $u\geq 0$, define
\[
\ell^*(u) := \sup_{\alpha\in\dR} \BRA{\alpha u - \ell_\alpha} = \frac{(u + \rc_X \rho + \rc_Y(1-\rho))^2}{2\PAR{\sigma_X^2 \eta^2 + \sigma_Y^2(\eta^2 - 2\rho + 1)}}.
\]
This function is continuous, and we have 
\[
\inf_{u>0} \frac{\ell^*(u^+)}{u} = \inf_{u>0} \frac{\ell^*(u)}{u} = \alpha_*.
\]

Denote $(-V)^*_n = \sup_{0\leq r < 1} (-V_{n+r})$. Then
\[
(-V)^*_n + V_n = \sup_{0\leq r<1} (-(V_{n+r} - V_n)) \leq |\rc_X| + |\rc_Y| + \sup_{0\leq r<1} Z_{n,r},
\]
where 
\[
Z_{n,r} = -\sigma_X \int_n^{n+r} F(s) \,\rd B^X_s - \sigma_Y \int_n^{n+r} (1-F(s)) \,\,\rd B^Y_s
\]
is a continuous centered Gaussian process with variance
\[
\sigma_X^2 \int_n^{n+r} F^2(s)\rd s + \sigma_Y^2 \int_n^{n+r} (1-F(s))^2 \,\rd s \leq \sigma^2, 
\quad \sigma^2 := \sigma_X^2 + \sigma_Y^2.
\]

Since $(\exp(\lambda Z_{n,r}))_{0\leq r\leq 1}$ is a submartingale for any $\lambda>0$, Doob's inequality gives
\[
\dP\PAR{\sup_{0\leq r\leq 1} Z_{n,r} > x} \leq \dE[\exp(\lambda Z_{n,1})] e^{-\lambda x} \leq \exp(\lambda^2 \sigma^2 / 2 - \lambda x).
\]
Optimizing over $\lambda$ yields
\[
\dP\PAR{\sup_{0\leq r\leq 1} Z_{n,r} > x} \leq \exp\PAR{-\frac{x^2}{2\sigma^2}}.
\]

Using $\dE[\re^{\alpha Z}] = 1 + \int_0^\infty \alpha e^{\alpha x} \dP(Z>x) \rd x$ for any nonnegative random variable $Z$, we obtain
\[
\dE\left[\exp\left(\alpha \sup_{0\leq r\leq 1} Z_{n,r}\right)\right] \leq 1 + \alpha \sqrt{2\pi} \sigma \,\exp(\tfrac{1}{2}\alpha^2 \sigma^2 ).
\]
Consequently, for any $\alpha>0$,
\[
\lim_{n\to \infty} \frac{1}{n} \log \dE[\exp(\alpha((-V)^*_n + V_n))] = 0.
\]

The decay rate of $p(u)$ then follows from \cite[Corollary 2.3]{duffield-oconnell95}.
\end{proof}

\begin{remark} \em
The ruin probability in our setting is closely linked to the fluctuations of a Lévy process relative to a curve. In particular, when $F$ is deterministic and the risk processes are arithmetic Brownian motions, the ruin probability can be expressed as the probability that a Brownian motion crosses a deterministic curve:
\[
p(u) = \dP\Big(\exists t \ge 0 : B_t < -u - b \circ a^{-1}(t)\Big),
\]
where $B$ is a standard Brownian motion, 
\[
a(t) = \sigma_X^2 \int_0^t F^2(s)\, \rd s + \sigma_Y^2 \int_0^t (1-F(s))^2\, \rd s, \qquad
b(t) = \rc_X \int_0^t F(s)\, \rd s + \rc_Y \int_0^t (1-F(s))\, \rd s.
\]
However, even in the case of Brownian motion, explicit formulas for general curves are not available. Relevant studies include \cite{AliliPatie10,AliliPatie14}, as well as the recent work \cite{Chaumont-Pellas23}, which provides an explicit expression for the probability that a Brownian motion crosses above a deterministic curve via its supremum, under a monotonicity assumption on the curve. We also mention \cite{HerrmannTanre16}, which proposes a numerical approach to the first-passage time of Brownian motion to a curved boundary.
\hfill $\Diamond$
\end{remark}

\section{Application: models under a SIS-dynamic}\label{sec:SIS}
The SIS ({\it Susceptible-Infected-Susceptible}) model describes the spread of an infectious disease in which individuals can become infected, recover, and subsequently return to the susceptible state without acquiring immunity. In such a population, susceptible individuals become infected through contact with infected individuals, while infected individuals recover at a given rate and immediately become susceptible again. In this section, we study a SIS-driven insurance risk model, for both the collective model (in which the SIS dynamic is deterministic) and the individual framework (in which the SIS dynamic is stochastic).

\medskip 

\subsection{Collective model}\label{sec:application collective model}

We now study the ruin probability for the collective model under various risk processes in a SIS epidemiological environment. Recall that, in the classical deterministic SIS epidemic model, the evolution of the fraction $F$ of infected individuals is governed by the ordinary differential equation (ODE)
\[
F'(t)=\beta (1-F(t))F(t)-\gamma F(t),
\]
where $\beta,\gamma>0$ are model parameters. Let $F_0$ denote the initial fraction of infected individuals. If $\beta\not=\gamma$, then the solution to the ODE is given by
\begin{equation}\label{eq:value of F in BM and SIS case}
F(t)=\dfrac{\beta -\gamma}{\PAR{\frac{\beta-\gamma}{F_0}-\beta}\re^{-(\beta-\gamma)t}+\beta} ,
\end{equation}
while if $\beta=\gamma$ then
\[
F(t)=\frac{F_0}{1+\beta F_0 t} .
\]

For future use, we can verify that if $\beta\neq\gamma$ then
\begin{align}
\int_0^tF(s)\rd s&= \frac{1}{\beta} \ln \PAR{1 + \frac{\beta F_0}{\beta-\gamma} \PAR{\re^{(\beta-\gamma)t}-1}} ,
\label{eq:expression int(F)} \\
\int_0^tF^2(s)\rd s&= \frac{\beta-\gamma}{\beta^2} \ln \PAR{1 + \frac{\beta F_0}{\beta-\gamma} \PAR{\re^{(\beta-\gamma)t}-1}} - \frac{F(t)-F_0}{\beta} , \label{eq:expression int(F2)}
\end{align}
and, if $\beta=\gamma$ then
\[
\int_0^tF(s)\rd s=\frac{1}{\beta}\ln\PAR{1+\beta F_0t} \quad \text{and} \quad \int_0^tF^2(s)\rd s=\frac{F_0^2 t}{1+\beta F_0 t}.
\]

The basic reproduction number, defined as the average number of infections generated by a single infected individual in a fully susceptible population, is given in this SIS model by $R_0=\frac{\beta}{\gamma}$. The following well-known asymptotic behavior holds:
\begin{itemize}
\item[$\circ$] If $R_0\leq1,\,\displaystyle\lim_{t\to\infty}F(t)=0.$
\item[$\circ$] If $R_0>1$ and $F_0>0,\,\displaystyle\lim_{t\to\infty}F(t)=1-\frac{1}{R_0}.$ 
\end{itemize}
It follows that $\rho$, as defined in \eqref{eq:def-rho}, is positive only when $\beta>\gamma$, i.e., when $R_0>1$. In this case,
\[
\rho=1-\frac{\gamma}{\beta}=1-\frac{1}{R_0}.
\]
Moreover, the parameter $\eta$ introduced in Assumption~\ref{A:hyp CLT for V} satisfies $\eta=\rho$.

In this setting, the process $X$ denotes the risk process associated with infected individuals and the process $Y$ denotes the risk process of susceptible individuals. We will study the probability of ruin for the case that $X$ and $Y$ are diffusive risk processes, and the case that they are both compound Poisson risk processes.

\medskip 
\paragraph{\it Diffusive risk processes}

We study the ruin probability for the collective model with arithmetic Brownian risk processes. 
From Equation~\eqref{eq:k alpha} with $\nu_X\equiv0\equiv\nu_Y$, the expression of $k_\alpha$ for the Brownian model can be easily deduced from~\eqref{eq:expression int(F)} and~\eqref{eq:expression int(F2)}.

The net profit condition in this setting is $\rho\rc_X+(1-\rho)\rc_Y>0$ and the decay rate $\alpha_*$ introduced in Proposition~\ref{prop: exact ruin brownian motion for collective model} is
\[
\alpha_*=2\frac{\rho\rc_X+(1-\rho)\rc_Y}{\rho^2\sigma_X^2+(1-\rho)^2\sigma_Y^2} .
\]

As $F$ is a monotone function, see Expression~\eqref{eq:value of F in BM and SIS case}, it takes its values exclusively either in the interval $[F_0,\rho)$ or in the interval $(\rho,F_0]$ (depending on the sign of the quantity $\frac{\beta-\gamma}{F_0}-\beta$). By Corollary~\ref{coro:estimate ruin spectrally positive}, the ruin probability satisfies
\[
p(u)\leq \exp\PAR{-\alpha_0 u}, \quad \text{with }\alpha_0:=2\min\PAR{\frac{F_0\rc_X+(1-F_0)\rc_Y }{F_0^2\sigma_X^2+(1-F_0)^2\sigma_Y^2},\frac{\rho\rc_X+(1-\rho)\rc_Y}{\rho^2\sigma_X^2+(1-\rho)^2\sigma_Y^2}}.
\]

For non class-dependent premiums, i.e.\ $\rc_X=\rc_Y$, the equality $\alpha_0=\alpha_*$ holds if either of the following conditions is satisfied:
 \begin{itemize}
 \item[$\circ$] $\rho \geq F_0$ and $F_0 + \rho \geq 2\sigma_Y^2 / (\sigma_X^2+\sigma_Y^2)$,
 \item[$\circ$] $\rho < F_0$ and $F_0 + \rho \leq 2\sigma_Y^2 / (\sigma_X^2+\sigma_Y^2)$.
 \end{itemize}
 Otherwise, we have $\alpha_0<\alpha_*$.

\medskip
We compare in Figure~\ref{fig:Ruin proba for the SIS Brownian model} the decay rates $\alpha_0$ and $\alpha_*$ with an empirical estimate of the ruin probability for this SIS-Brownian model. To compute the estimate, we used a Monte-Carlo method (with $5000$ iterations) and an Euler scheme (with step size $0.1$) with the following parameters: $\beta=2$, $\gamma=1$, $\rc_X=\rc_Y=0.1$, $\sigma_X=1$, $\sigma_Y=2$ and different values of $F_0$. For this set of parameters, we have $\alpha_*=0.16$. Note that when $F_0=\rho$, the function $F$ is constant.

\begin{figure}[h!]
    \centering
    \begin{tabular}{ccc}
    \includegraphics[scale=0.3]{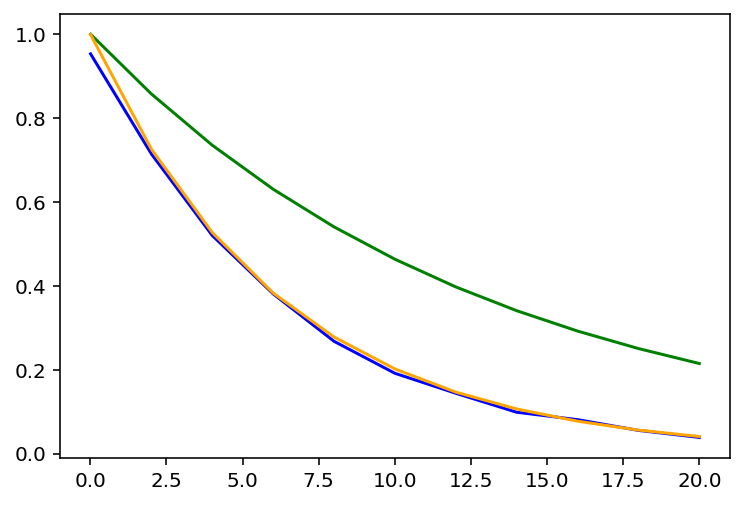}& \includegraphics[scale=0.3]{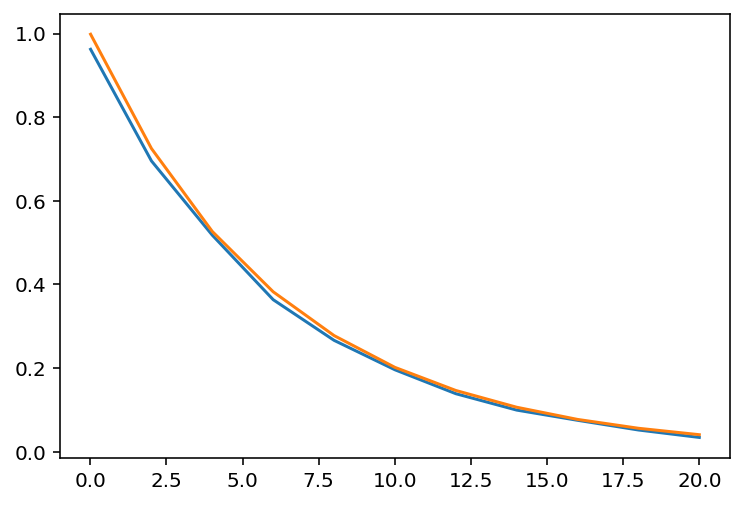}&
    \includegraphics[scale=0.3]{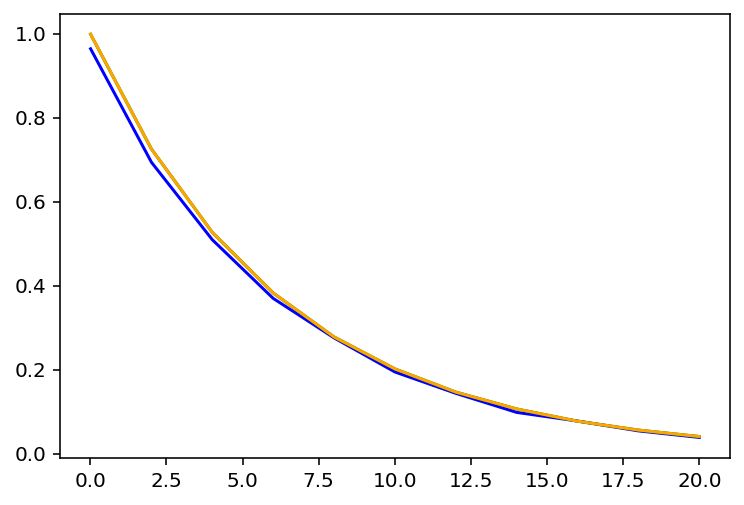}\\
    when $F_0=0.2<\rho$&
    when $F_0=0.5=\rho$ &
    when $F_0=0.7>\rho$ \\
     ($\alpha_*>0.077\simeq\alpha_0$)&($\alpha_*=\alpha_0$)&
     ($\alpha_*=\alpha_0$)
    \end{tabular}
    \caption{Comparison of an estimate of the ruin probability before time $T=1000$ (in blue) with the curves $u\mapsto\re^{-\alpha_0 u}$ (in green) and $u\mapsto\re^{-\alpha_* u}$ (in orange) for the SIS-Brownian model with an initial surplus $u\in[0.1,20]$ (on the horizontal axis). 
    }
    \label{fig:Ruin proba for the SIS Brownian model}
\end{figure}

\medskip
By Proposition~\ref{prop: exact ruin brownian motion for collective model}, $\alpha_*$ is the decay rate at infinity of the ruin probability. At first glance, based on the identity satisfied by the ruin probability in Proposition~\ref{prop: exact ruin brownian motion for collective model} and on the shape of Figure~\ref{fig:Ruin proba for the SIS Brownian model}, one might be led to believe that a Lundberg inequality of the form $p(u)\leq \re^{-\alpha_* u}$ holds here. However, this is not always the case for the Brownian risk model in a heterogeneous time-dependent  population. 
 In fact, with the same set of parameters as in Figure~\ref{fig:Ruin proba for the SIS Brownian model}, we observe in Figure~\ref{fig:Root theta(t)} the convergence of  the positive root $\alpha(t)$ of $\alpha\mapsto k_{\alpha}(t)$ to $\alpha_*$ when $t\to\infty$ (see Proposition~\ref{prop:convergence of the roots}). Note that in this specific model, \[\alpha(t)=2\frac{\rc_X\int_0^t F(s)\,\rd s+\rc_Y\int_0^t (1-F(s))\,\rd s}{\sigma_X^2\int_0^t F^2(s)\,\rd s+\sigma_Y^2\int_0^t (1-F(s))^2\,\rd s}\] has an explicit expression, thanks to  \eqref{eq:expression int(F)} and \eqref{eq:expression int(F2)}. In particular, when $F_0=0.2$, we have $\alpha(t)< \alpha_*$, which means that $\forall t> 0$, $k_{\alpha_*}(t)>0$ and then $p(u)> \exp\PAR{-\alpha_* u}$ by Proposition~\ref{prop: exact ruin brownian motion for collective model}. The opposite behavior occurs when $F_0=0.7$.
We deduce from this example that Lundberg's inequality does not hold for certain values of the deterministic SIS model parameters, and thus for the collective model~\eqref{eq:def-V} in general. We conjecture that for a range of parameters (the same implying $\alpha_0=\alpha_*$), the Lundberg inequality $p(u)\leq \re^{-\alpha_* u}$ holds in the Brownian-SIS model, and that for the remaining parameter values, there is a function $g$ with $g(0)=1$ and $\lim_{u\to \infty}\frac{\log g(u)}{u}=0$ such that for all $u> 0$,
\[
\re^{-\alpha_* u}<p(u)\leq g(u)\re^{-\alpha_* u}.
\]

\begin{figure}[h!]
    \centering
    \includegraphics[scale=0.3]{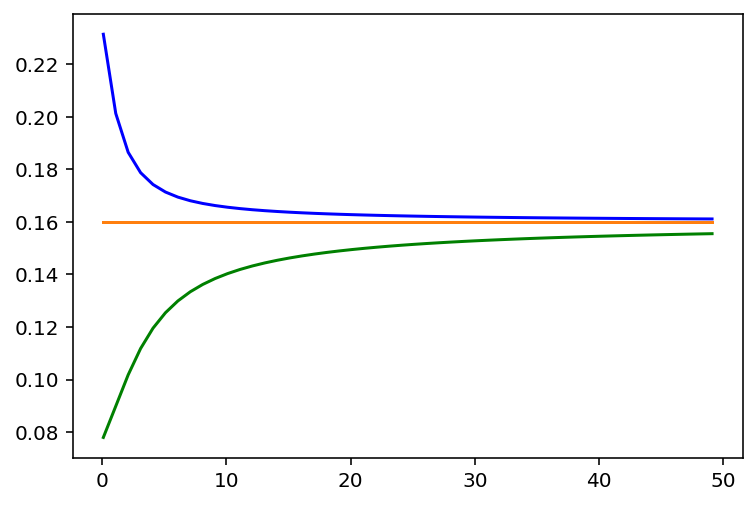}
    \caption{The functions $t\mapsto\alpha(t)$ when $F_0=0.2$ (in green) and when $F_0=0.7$ (in blue), and the value $\alpha_*$ in orange.
    }
    \label{fig:Root theta(t)}
\end{figure}

\bigskip

\paragraph{\it Compound Poisson risk processes}

Assume now  that $X$ and $Y$ are independent compound Poisson processes with drift $\rc_X$ and $\rc_Y$, Poisson rate $\lambda_X$ and $\lambda_Y$ and (downward) exponential jumps with rate $\delta_X$ and $\delta_Y$, respectively. 
The process $V$, defined by \eqref{eq:def-V}, can be written in this case in the following way
\begin{align}\label{eq:V in the compound poisson case}
V(t)=\rc_X\int_0^tF(s)\rd s+\rc_Y\int_0^t(1-F(s))\rd s-\sum_{i=1}^{N_t^X}F(T_i^{X})A_i-\sum_{i=1}^{N_t^Y}\PAR{1-F(T_i^{Y})}B_i,
\end{align}
where $N^X$ is the Poisson counting process associated with the rate $\lambda^X$, $(T^X_i)_{i\geq 1}$ are the successive jump times of $N^X$, and $(A_i)_{i\geq 1}$ are i.i.d.\ exponential r.v.\ with the rate $\delta_X$ (and similarly for $N^Y$,  $(T^Y_i)_{i\geq 1}$ and $(B_i)_{i\geq 1}$ associated with the process $Y$). We note from Expression~\eqref{eq:V in the compound poisson case} that as soon as $\rc_X\geq 0$ and $\rc_Y\geq 0$, the drift function is non-decreasing and thus ruin can only occur at a jump time of the counting processes $N^X$ and $N^Y$.

The Laplace exponent of $X$ is given by $\log\dE\SBRA{\re^{\alpha X(1)}} = \rc_X \alpha - \lambda_X \frac{\alpha}{\delta_X+\alpha}$, for $\alpha>-\delta_X$; we have a similar expression for the Laplace exponent of $Y$. As a consequence, for $\alpha < \min \left\{\delta_X/f_\mathrm{max}, \delta_Y/(1-f_\mathrm{min}) \right\}$,
\begin{multline*}
k_\alpha (t) = \alpha \left[ -\rc_X \int_0^t F(s) \mathrm{d}s - \rc_Y \int_0^t \PAR{1-F(s)} \mathrm{d}s \right. \\
\left. + \lambda_X \int_0^t \frac{F(s)}{\delta_X - \alpha F(s)} \mathrm{d}s + \lambda_Y \int_0^t \frac{1-F(s)}{\delta_Y - \alpha \PAR{1-F(s)}} \mathrm{d}s \right] .
\end{multline*}

In the deterministic SIS-dynamic, we compute for $\alpha<\overline{\alpha}$
\[
\ell_\alpha := \lim_{t \to \infty} \frac{k_\alpha (t)}{t} =\alpha \left[ -\PAR{\rho \rc_X + (1-\rho) \rc_Y }+ \frac{\rho \lambda_X}{\delta_X - \alpha \rho} + \frac{(1-\rho) \lambda_Y}{\delta_Y - \alpha (1-\rho)} \right],
\]
with $\rho=1-\frac{\gamma}{\beta}$, $\overline{\alpha}=\min\BRA{\frac{\delta_X}{\rho},\frac{\delta_Y}{1-\rho}}$ and the net profit condition is
\begin{equation}\label{eq:net profit condition for CPP}
\rho \dE \SBRA{X_1} + \PAR{1-\rho} \dE \SBRA{Y_1}=\rho\PAR{\rc_X-\frac{\lambda_X}{\delta_X}} +(1-\rho)\PAR{\rc_Y-\frac{\lambda_Y}{\delta_Y}}> 0 .
\end{equation}

We compare in Figure~\ref{fig:Ruin proba for the SIS CPP}, an empirical estimate of the ruin probability, via a Monte-Carlo method (with $1000$ iterations), the rough estimate $\alpha_0$ from Proposition~\ref{prop:estimation ruin probability for the collective model} and the rate $\alpha_*$ from Theorem~\ref{thm: exact ruin  for the collective model}, in a deterministic SIS environment (with $\beta=2,\,\gamma=1$) for different values of $F_0$ and for the following parameters: $\lambda_X=1$, $\delta_X=1$, $\rc_X=\rc_Y=3$, $\lambda_Y=2$ and $\delta_Y=1/2$. Since $F$ is monotone and converges to $\rho$, and from their  respective definitions in Proposition~\ref{prop:estimation ruin probability for the collective model} and in Theorem~\ref{thm: exact ruin  for the collective model}, $\alpha_0$ is the solution to 
\[
\alpha_0=\sup\BRA{\alpha>0:C_\alpha(F_0)<0 \text{ and } C_\alpha(\rho)<0} 
\]
and $\alpha_*=\sup\BRA{\alpha>0: C_\alpha(\rho)<0}$, 
with
\begin{align*}
C_\alpha(z)=&
-(z\rc_X +(1-z)\rc_Y)+ \frac{z \lambda_X}{\delta_X - \alpha z} + \frac{(1-z) \lambda_Y}{\delta_Y - \alpha (1-z)}.
\end{align*}
Solving the second-degree equation $C_\alpha(z)=0$ in $\alpha$ on the open interval $\PAR{0,\bar{\alpha}}$, we deduce that, under the net profit condition \eqref{eq:net profit condition for CPP} and the condition $F_0\dE[X_1]+(1-F_0)\dE[Y_1]>0$, 
\[
\alpha_0=\min(r(F_0),r(\rho))\quad \text{ and }\quad\alpha_*=r(\rho),
\]
with
\begin{align*}
r(z)=\frac{\rc(\delta_X(1-z)+\delta_Yz)-z(1-z)(\lambda_X+\lambda_Y)-\sqrt{\Delta}}{2\rc z(1-z)},
\end{align*}
$\rc=z\rc_X +(1-z)\rc_Y$, and 
\[
\Delta=\big[\rc(\delta_X(1-z)+\delta_Y z)-z(1-z)(\lambda_X+\lambda_Y)\big]^2-4\rc z(1-z)\big[\rc\delta_X\delta_Y-z\lambda_X\delta_Y-(1-z)\lambda_Y\delta_X\big].
\]
For the set of parameters used in Figure~\ref{fig:Ruin proba for the SIS CPP}, we compute  $\alpha_*\simeq 0.183$. In particular, we can see in Figure~\ref{fig:Ruin proba for the SIS CPP} how inaccurate the rough estimate can be in some cases.
\begin{figure}[h!]
    \centering
    \begin{tabular}{ccc}
    \includegraphics[scale=0.3]{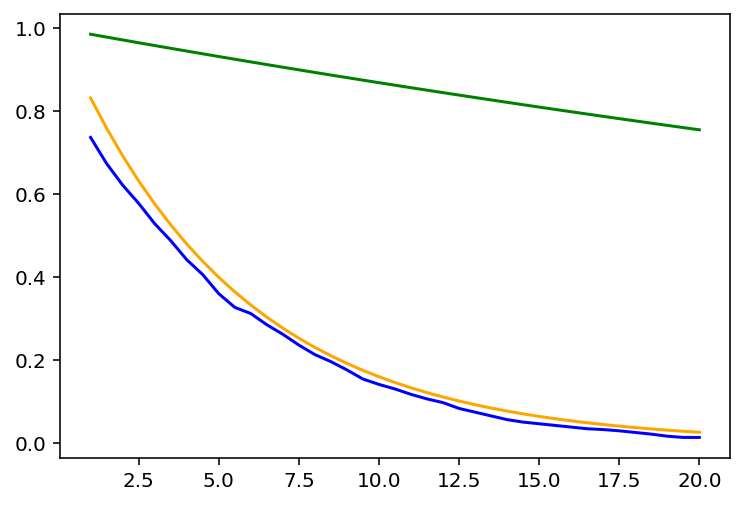}& \includegraphics[scale=0.3]{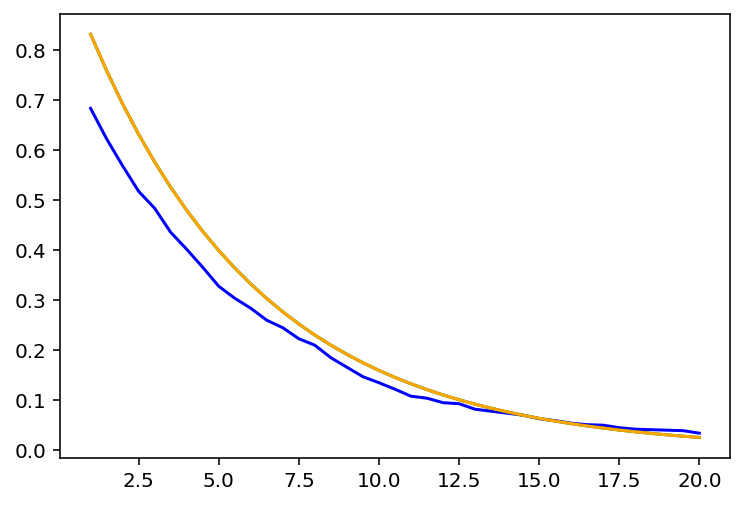}&
    \includegraphics[scale=0.3]{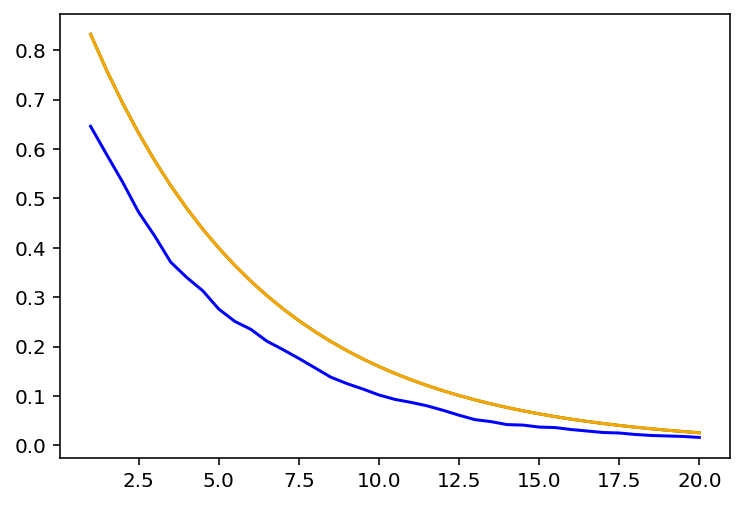}\\
    when $F_0=0.35<\rho$&
    when $F_0=0.5=\rho$ &
    when $F_0=0.7>\rho$ \\
     ($\alpha_*>0.014=\alpha_0$)&($\alpha_*=\alpha_0$)&
     ($\alpha_*=\alpha_0$)
    \end{tabular}
    \caption{Comparison of an estimate of the ruin probability before time $T=1000$ (in blue) with the curves $u\mapsto\re^{-\alpha_0 u}$ (in green) and $u\mapsto\re^{-\alpha_* u}$ (in orange) for the SIS-Compound Poisson model with an initial surplus $u\in[1,20]$ (on the horizontal axis). 
    }
    \label{fig:Ruin proba for the SIS CPP}
\end{figure}

\subsection{Individual model}\label{sec:application individual model}

We consider the individual model~\eqref{eq:def-W} and define the ruin probability by
\begin{equation}\label{eq:def ruin proba W}
q(u):=\dP\PAR{u+\inf_{t\geq 0}W(t)<0},
\end{equation}
when the insurance firm’s initial reserve is equal to $u$. The ruin time for the individual model is denoted by
\begin{equation}\label{eq:def ruin time W}
\zeta(u)=\inf\BRA{t\geq 0: u+\inf_{t\geq 0}W(t)<0}.
\end{equation}

In this subsection, we point out how the decay rate of $q(u)$ can be calculated in the case where $I(\cdot)$ is an irreducible continuous-time Markov process on the state space $\{0,\ldots,N\}$, for instance a birth-death process, and 
in which $X$ and $Y$ are either both spectrally negative or both spectrally positive L\'evy processes.

Let the transition rate matrix $Q$ underlying the process $I$ be such that the rates are denoted by $Q_{ij}$, with $Q_i:=-Q_{ii}=\sum_{j\not=i}Q_{ij}.$ In a SIS-epidemiological setting, we have  
\[
Q_{i,i+1}= \beta i (N-i)+\lambda,\quad Q_{i,i-1}= \gamma i .
\]
Note that the `external arrival' (immigration) rate $\lambda>0$ has been added to make sure the SIS process is not absorbed in state $0$. 
We denote by $\psi_X$ and $\psi_Y$ the Laplace exponents of $X$ and $Y$ respectively, i.e.,
    \[
    \dE\SBRA{\re^{\alpha X(t)}}=\re^{\psi_X(\alpha)t}\quad \text{and} \quad \dE\SBRA{\re^{\alpha Y(t)}}=\re^{\psi_Y(\alpha)t}.
    \]
In this setup, the process $W(\cdot)$ is a {\it Markov-additive process} (MAP); see e.g.\ \cite[Section III.5]{Asmussen-Albrecher10}.

\medskip
    
\paragraph{\it Spectrally negative L\'evy processes}

We assume that the L\'evy processes $X$ and $Y$ are light-tailed, in the sense  that the {\it Laplace exponents} of $X$ and $Y$, denoted by $\psi_X(\cdot)$ and $\psi_Y(\cdot)$, respectively are well-defined for arguments in an open neighborhood of $0$. 
By 
Lemma \ref{lem:characteristic function W}, we have that
\[
\xi_\alpha(t):=\frac{1}{t}\log \dE\SBRA{\re^{-\alpha W(t)}}={\psi_Y(-\alpha)N}+\frac{1}{t}\log \dE\SBRA{\exp\left( \big(\psi_X(-\alpha)-\psi_Y(-\alpha)\big)\int_0^tI(s)\rd s\right)}. \]
As for the collective model, the key assumption is that the equation
\[h_\alpha :=\lim_{t\to\infty} \xi_\alpha(t) = 0\]
 has a positive root $\alpha_*$, which is, under the net profit condition, necessarily unique due to the convexity of $\alpha\mapsto h_\alpha$. 
By   e.g.\ \cite{BotvichOConnell1995}, the decay rate of the ruin probability equals $-\alpha_*$ (which is a negative number):
\[\lim_{u\to\infty}\frac 1 u \log q(u)\to -\alpha_*.\]

It is not immediately clear how to compute $\xi_\alpha(\cdot)$ (and thus $\alpha_*$), but it turns out that there is a convenient approach to do so. 
\begin{itemize}
    \item[$\circ$]

Let $\Psi(\alpha)$ be the $(N+1)\times(N+1)$ diagonal matrix whose $i$-th diagonal entry equals $i\psi_X(-\alpha)+(N-i)\psi_Y(-\alpha)$, for $i=0,\ldots,N$. 
By the results of Van Kreveld {\it et al.}~\cite{vanKreveldMandjesDorsman2024}, $\alpha_*$ is the value of $\alpha$ with largest negative real part such that $\det(Q+\Psi(\alpha))=0$; note that in \cite{vanKreveldMandjesDorsman2022} the focus is on exceeding a {\it positive} level, so that all L\'evy processes considered there have to be multiplied by $-1$.

\item[$\circ$]Fix a reference state $i_0\in\{0,\ldots,N\}$; we will later argue that the choice of $i_0$ is immaterial. 
Let $\bM(\cdot)$ denote the moment generating function of the increment of $W(\cdot)$ between two successive returns of the modulating process $I(\cdot)$ to the state $i_0$. 
A conditioning argument shows that $\bM(\cdot)$ can be computed by solving the linear system
\begin{align*}
\bM(\alpha) 
&= \int_0^\infty Q_{i_0}\re^{-Q_{i_0}t}\,\re^{(i_0\psi_X(-\alpha)+(N-i_0)\psi_Y(-\alpha))t}\,\rd t 
\sum_{j\neq i_0}\frac{Q_{i_0 j}}{Q_{i_0}}\,\bM_j(\alpha),\\
\bM_i(\alpha) 
&= \int_0^\infty Q_i\re^{-Q_i t}\,\re^{(i\psi_X(-\alpha)+(N-i)\psi_Y(-\alpha))t}\,\rd t 
\left(\sum_{j\neq i,i_0}\frac{Q_{ij}}{Q_i}\,\bM_j(\alpha)+\frac{Q_{i i_0}}{Q_i}\right),
\end{align*}
where the second equation holds for any $i\neq i_0$. Here, $\bM_i(\alpha)$ denotes the moment generating function of the increment of $W(\cdot)$ when the modulating process starts in state $i$ and is observed until it hits $i_0$.
These equations simplify to
\begin{align}
\left(Q_{i_0}-i_0\psi_X(-\alpha)-(N-i_0)\psi_Y(-\alpha)\right)\bM(\alpha)
&= \sum_{j\neq i_0} Q_{i_0 j}\,\bM_j(\alpha), \label{d1}\\
\left(Q_i-i\psi_X(-\alpha)-(N-i)\psi_Y(-\alpha)\right)\bM_i(\alpha)
&= \sum_{j\neq i,i_0} Q_{ij}\,\bM_j(\alpha)+Q_{i i_0}, \label{d2}
\end{align}
where again the second equation applies for all $i\neq i_0$.

\item[$\circ$] The crucial insight is that, again following \cite{vanKreveldMandjesDorsman2024}, $\alpha_*$ is the (unique) positive solution of  $\bM(\alpha)=1$ (cf.\ the {\it Lundberg equation}). Hence we can compute $\bM(\alpha)$ for given $\alpha$ by solving the above linear equations; by bisection we can find $\alpha_*$ as the $\alpha$ such that $\bM(\alpha)=1$. It is readily checked that the choice of the reference state $i_0$ has no impact; observe that $ \bM(\alpha_*)$ is a right eigenvector of $Q+\Psi(\alpha_*)$, which was normalized such that $\bM_{i_0}(\alpha_*)=1$, and which is componentwise positive due to its interpretation in terms of moment generating functions. 
\end{itemize}

As it turns out, it is actually possible to find a sharper result. In \cite[Thm.\ 5.1]{vanKreveldMandjesDorsman2024}, it is pointed out how to compute the number $\theta\in(0,\infty)$ such that, as $u\to\infty$,
\[q(u) \,\re^{-\alpha_* u} \to \theta.\]
We do not provide the precise details underlying the computation of $\theta.$ Notably, the expressions provided in \cite{vanKreveldMandjesDorsman2024} do not considerably simplify if the process $I(\cdot)$ corresponds to a birth-death process (but obviously there is the reduced complexity of solving the linear system for fixed $\alpha$). We refer to Figure~\ref{fig:Ruin proba for the SIS BM} for a numerical illustration. 

\begin{figure}[h!]
    \centering
    \begin{tabular}{cc}
    \includegraphics[scale=0.45]{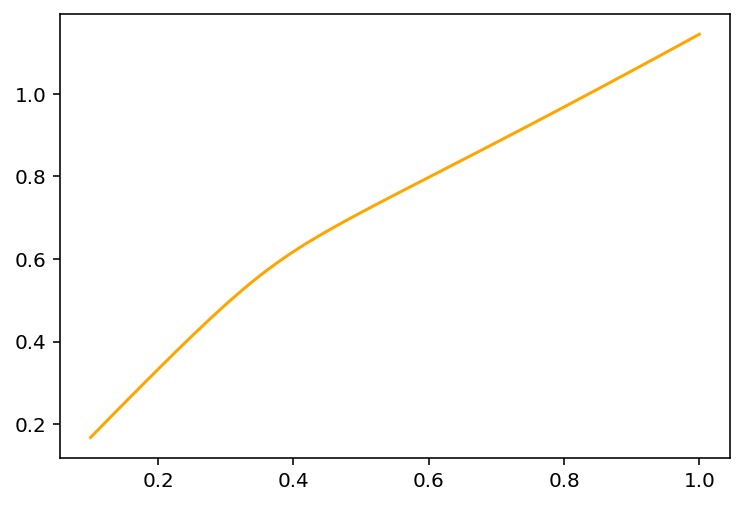}&
    \includegraphics[scale=0.45]{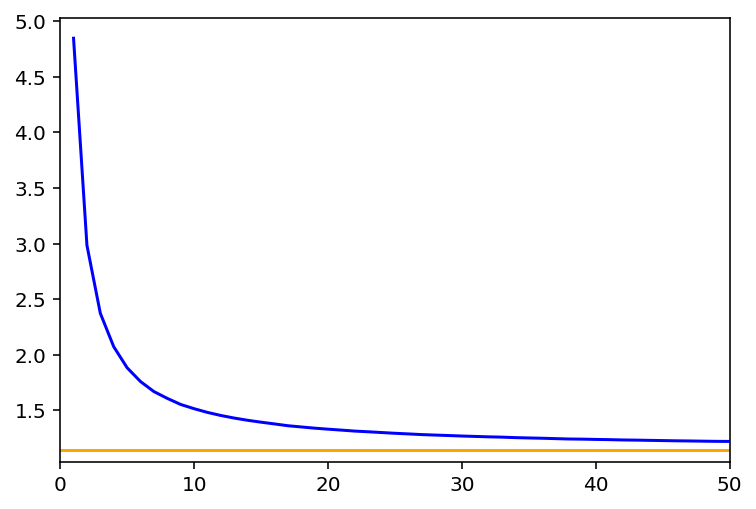}\\
    \end{tabular}
    \caption{Left: the function $\rc \mapsto \alpha^*(\rc)$ is considered in the individual model within the SIS epidemiological setting, where $X$ and $Y$ are Brownian motions with equal drift $\rc \in [0.1,1]$ and variances $\sigma_X^2 = 1$ and $\sigma_Y^2 = 2$, respectively. Right: Comparison of Monte Carlo-based estimates for $\log q(u)/u$ in the individual model with $\rc=1$, for initial surplus $u \in [1,50]$ (in blue), versus the asymptotic decay rate $\alpha^* \simeq 1.14418$ (in orange). The ruin probability estimates were obtained using importance sampling, as described in~\cite[Section 6]{vanKreveldMandjesDorsman2024}. The parameters used are $N=5$, $\beta=3$, $\gamma=2$, and $\lambda=1$.
    }
    \label{fig:Ruin proba for the SIS BM}
\end{figure}

\bigskip
\paragraph{\it Spectrally positive L\'evy processes}
We now consider the case that $X$ and $Y$ are spectrally positive L\'evy processes. This brings us in the setting of \cite[Section 4]{vanKreveldMandjesDorsman2022}.
In this case, we define $\Psi(\alpha)$ as the  $(N+1)\times(N+1)$-dimensional diagonal matrix whose $i$-th diagonal entry is $i\psi_X(\alpha) +(N-i)\psi_Y(\alpha)$ for $i=0,\ldots,N$, 
where $\psi_X$ and $\psi_Y$ are the Laplace exponents of $X$ and $Y$, respectively.  

Under the mild assumption that the $N+1$ roots of ${\rm det}(Q+\Psi(\alpha))$ with positive real parts are unique, one can compute $q(u)$ as follows. Denote these roots by $\nu_0, \ldots, \nu_N$. A crucial observation is that the first-passage process associated with $W(\cdot)$ in the negative direction (i.e., the process $(\zeta(u))_{u\geqslant 0}$ defined by~\eqref{eq:def ruin time W}), together with the state of the background process, forms a MAP — just like $W(\cdot)$ itself is a MAP.
As a consequence, $-\inf_{t\geqslant 0}W(t)$ has a phase-type distribution.
As pointed out in detail in \cite[Section 4]{vanKreveldMandjesDorsman2022}, for a $(N+1)\times (N+1)$ transition rate matrix $\Lambda$, a vector 
${\boldsymbol \lambda}:=-\Lambda{\boldsymbol 1}\geqslant {\boldsymbol 0}$ with at least one positive entry, and initial distribution ${\boldsymbol a}$, we have that 
\[{\mathbb E}\big[ \re^{\gamma \inf_{t\geqslant 0}W(t)}\big] = {\boldsymbol a}(\gamma I - \Lambda)^{-1}{\boldsymbol \lambda}.\]
It is then noted that the zeroes of ${\rm det}(Q+\Psi(\nu))$ coincide with those of ${\rm det}(- \nu I-\Lambda )$; 
cf.\ \cite[Thm.\ 4.7]{Ivanovs2011OnesidedMAP} and the first statement of \cite[Corollary 4.21]{Ivanovs2011OnesidedMAP}.
Because of this, the matrix $(\gamma I - \Lambda)$ is singular in $\gamma=-\nu_0,...,-\nu_N$, hence ${\mathbb E}[ \re^{\gamma \inf_{t\geqslant 0}W(t)}]$ can be written as a linear combination of the terms $1/(\gamma+\nu_0),...,1/(\gamma+\nu_N)$ via a partial fraction expansion.
This means that  we can write, for  $u\geqslant 0$,
\begin{equation}\label{dec}q(u) = \sum_{k=0}^N c_{k} \,e^{-\nu_k u}.\end{equation}
In \cite[Thms.\ 3 and 4]{vanKreveldMandjesDorsman2022} a recipe is presented for computing the coefficients $c_0,\ldots,c_N.$


\appendix

\renewcommand{\theequation}{\Alph{section}.\arabic{equation}}
\let \sappend=\section
\renewcommand{\section}{\setcounter{equation}{0}\sappend}

\section{Proof of Proposition~\ref{prop:equality in distribution of $W$}}\label{A:proof equality in distribution of $W$}

By Lemma~\ref{lem:characteristic function W}, we easily deduce that for any given $t\geqslant 0$, $W(t)\overset{\rm {\small d}}{=}U(t)$. To prove the equality of the processes in distribution, we need to prove the equality of the finite-dimensional probability distributions. Consider, for $n\in{\mathbb N}$, a vector $\PAR{t_1,\ldots, t_n}\in\dR_+^n$. Without loss of generality, we can assume that $t_1\leq \ldots\leq t_n$. Consequently, $J(t_1)\leq \ldots\leq J(t_n)$ a.s, where $J(t)=\int_0^tI(s)\,\rd s$.

Let $(\alpha_1,\ldots,\alpha_n)\in\dR^n$, and define $\beta_k:=\sum_{\ell=k}^n\alpha_\ell$. Introducing $t_0=0$, we have
\begin{align*}
\exp\left({i\sum_{k=1}^n\alpha_k U(t_k)}\right)&=\exp\left({i\sum_{k=1}^n\beta_k\, \PAR{U(t_k)-U(t_{k-1})}}\right),
\end{align*}
as $U(0)=X(0)+Y(0)=0$.
Recalling that $\cF^I$ is the $\sigma$-algebra generated by $I$, we have that, conditionally on this $\cF^I$, the random variables $\PAR{U(t_k)-U(t_{k-1})}_{1\leq k\leq n}$ are independent, and 
\[U(t_k)-U(t_{k-1})\overset{\rm {\small d}}{=}U\PAR{J(t_k)-J(t_{k-1})}.\] 
As a consequence,
\begin{align}
\dE\SBRA{\exp\left({i\sum_{k=1}^n\alpha_k U(t_k)}\right)\,\Big\vert\, \,\cF^I}&=\prod_{k=1}^{n}\dE\SBRA{\exp\left({i\beta_k \PAR{U(t_k)-U(t_{k-1})}}\right)\,\Big\vert \,\cF^I}\notag\\
&\hspace{-25mm}=\prod_{k=1}^n\exp\left({\varphi_X\PAR{\beta_k }\int_{t_{k-1}}^{t_k}I(s)\rd s+\varphi_Y\PAR{\beta_k}\int_{t_{k-1}}^{t_k} \PAR{N-I(s)}\rd s}\right).\label{eq:conditional characteristique U}
\end{align}
On the other hand, we have
\begin{align*}
\exp\left({i\sum_{k=1}^n\alpha_k W(t_k)}\right)=\exp\left({i\sum_{k=1}^n\beta_k \PAR{W(t_k)-W(t_{k-1})}}\right),
\end{align*}
where
$
W(t)= \sum_{j=1}^{N}W^j(t)$
with
\[
    W^j(t)=\int_0^t\ind_\BRA{\chi^j(s-)=\rA}\rd X^j(s)+\int_0^t\PAR{1-\ind_\BRA{\chi^j(s-)=\rA}}\rd Y^j(s).
\]
Recall that $X$ and $Y$ are independent Lévy processes, which are in addition  independent of $\chi$. With $\cF^\chi$ the $\sigma$-algebra generated by $\chi=\PAR{\chi^1,\ldots, \chi^N}$, recalling that the processes $(W^j)_{1\leq j\leq N}$ are independent conditionally on $\cF^\chi$, and the increments of the integrals are independent conditionally on $\cF^\chi$, we conclude
\begin{align}
&\dE\SBRA{\exp\left({i\sum_{k=1}^n\alpha_k W(t_k)}\right)\,\Big\vert \,\cF^\chi}\\&=\prod_{j=1}^N\prod_{k=1}^n\dE\SBRA{\exp\left({i\beta_k \PAR{\int_{t_{k-1}}^{t_k}\ind_\BRA{\chi^j(s^-)=\rA}\rd X^j(s)+\PAR{1-\ind_\BRA{\chi^j(s^-)=\rA}}\rd Y^j(s)}}\right)\,\Big\vert \,\cF^\chi}
\notag\\
&=\prod_{j=1}^N\prod_{k=1}^n\exp\left({\varphi_X(\beta_k) \int_{t_{k-1}}^{t_k}\ind_\BRA{\chi^j(s)=\rA}\rd s+\varphi_Y(\beta_k)\int_{t_{k-1}}^{t_k}\PAR{1-\ind_\BRA{\chi^j(s)=\rA}}\rd s}\right)
\notag\\
&=\prod_{k=1}^n\exp\left({\varphi_X(\beta_k) \int_{t_{k-1}}^{t_k}I(s)\,\rd s+\varphi_Y(\beta_k)\int_{t_{k-1}}^{t_k}\PAR{N-I(s)}\,\rd s}\right).\label{eq:conditional characteristique W}
\end{align}
Here we used in the above second line an argument similar to the one relied upon in the proof of  Lemma~\ref{lem:characteristic function V}, and  in the last line \[I_N(s)=\sum_{j=1}^N \ind_\BRA{\chi^j(s)=\rA}.\] The conclusion follows by taking the expectation in \eqref{eq:conditional characteristique U} and \eqref{eq:conditional characteristique W}.

\section{Proof of Lemma~\ref{lem:esp-variance}}
\label{B:proof means and variances}
To lighten the notations, we suppress the superscripts $\rc$ and $\ri$ in the proof. 
We compute the two first derivatives of the characteristic function of $V(t)$, given by~\eqref{eq:characteristic function V}, in $\alpha=0$. First, we have
\[
i\,\dE[V(t)]=\dE\SBRA{\int_0^t\PAR{\varphi_X'(0)F(s)+\varphi_Y'(0)(1-F(s))}\rd s},
\]
and we easily deduce the mean of $V(t)$.

Then, we have
\begin{multline*}
-\dE[V^2(t)]=\dE\SBRA{\int_0^t\PAR{\varphi_X''(0)F^2(s)+\varphi_Y''(0)\PAR{1-F(s)}^2}\rd s}\\
+\dE\SBRA{\PAR{\int_0^t\PAR{\varphi_X'(0)F(s)+\varphi_Y'(0)(1-F(s))}\rd s}^2}.
\end{multline*}
Consequently
\begin{align*}
\var(V(t))=&-\varphi_X''(0)\dE\SBRA{\int_0^t F^2(s)\rd s}-\varphi_Y''(0)\dE\SBRA{\int_0^t\PAR{1-F(s)}^2\rd s}\\
&-\PAR{\varphi'_X(0)}^2\,\var\PAR{\int_0^t F(s)\rd s}-\PAR{\varphi'_Y(0)}^2\,\var\PAR{\int_0^t(1- F(s))\rd s}\\
&-2\varphi'_X(0)\varphi'_Y(0)\,\cov\PAR{\int_0^t F(s)\rd s,\int_0^t (1-F(s))\rd s}\\
=&-\varphi_X''(0)\,\dE\SBRA{\int_0^t F^2(s)\rd s}-\varphi_Y''(0)\,\dE\SBRA{\int_0^t\PAR{1-F(s)}^2\rd s}\\
&-\PAR{\varphi'_X(0)-\varphi'_Y(0)}^2\,\var\PAR{\int_0^t F(s)\rd s}\\
=&\,\var(X(1))\, \dE\SBRA{\int_0^t F^2(s)\rd s}+\var(Y(1))\,\dE\SBRA{\int_0^t\PAR{1-F(s)}^2\rd s}\\
&+\PAR{\dE[X(1)]-\dE[Y(1)]}^2\,\var\PAR{\int_0^t F(s)\rd s},
\end{align*}
and the expression of $\var(V(t))$ follows.

  The mean and variance of $W(t)$ can be either computed from \eqref{eq:def-W} and the expressions of the mean and variance of $V$ applied to $W^j$, or  from the equality in distribution in Proposition~\ref{prop:equality in distribution of $W$} and  from Wald's equation for means and variances:
\begin{align*}
        \dE\SBRA{W(t)}&=\dE[X(1)]\dE\SBRA{J(t)}+\dE[Y(1)]\dE\SBRA{Nt-J(t)},\\
        \var{(W(t))}&=\var{(X(J(t)))}+2\,{\rm Cov}\left(X(J(t)),Y(Nt-J(t))\right)+\var{\left(Y(Nt-J(t))\right)},
        \end{align*}
where $J(t)=\int_0^tI(s)\rd s$ and
\begin{align*}   
\var\PAR{X(J(t))}&=\var(X(1))\dE\SBRA{J(t)}+\dE[X(1)]^2 \,\var{(J(t))},\\
\var\PAR{Y(Nt-J(t))}&=\var(Y(1))\dE\SBRA{Nt-J(t)}+\dE[Y(1)]^2 \,\var{(J(t))},
\end{align*}
and (using that $X$ and $Y$ are independent)
    \begin{align*}
        {\rm Cov}\left(X(J(t)),Y(Nt-J(t))\right)&=\dE[X(1)]\dE[Y(1)]\,{\rm Cov}\left(J(t),Nt-J(t)\right)\\
        &=-\dE[X(1)]\dE[Y(1)]\,{\rm Var}\left(J(t)\right).
    \end{align*}
    Upon combining the above, we obtain the result.

\section{Proof of Lemma~\ref{lem: Martingale}}\label{Appendix:martingale}

For each $\alpha\in(0,\overline{\alpha}]$, it is easy to conclude that $M_\alpha(t)=\re^{-\alpha V(t)-k_\alpha(t)}$ is well defined and $\cF_t$-measurable, for all $t \geq 0$. In fact, for any $t\geq 0$, as $F(t)\in[0,1]$ and $\re^{-\overline{\alpha} X(1)}$ is integrable, we deduce that $\re^{-\alpha F(t)X(1)}$ is also integrable. Consequently, since $\re^{\varphi_X(i\alpha F(t))}=\dE\SBRA{\re^{-\alpha F(t)X(1)}\,\vert\, \cF^F}$, $\varphi_X(i\alpha F(t))$ is well defined. Similarly, $\varphi_Y(i\alpha F(t))$ and $k_\alpha(t)$ (given by~\eqref{eq:k alpha}) are well defined, for each $\alpha\in(0,\overline{\alpha}]$ and each $t\geq 0$.

Note that $M_\alpha$ is càdlàg since $X,Y,F$ are themselves càdlàg by assumption. The integrability of $M_\alpha(t)$ is a consequence of the relationship $\re^{k_\alpha(t)}=\dE\SBRA{\re^{-\alpha V(t)}\,\vert\, \cF^F}$.

Now, take $s,t\geq 0$. We have                         
\begin{align}
&\dE\left[M_\alpha(t+s)\big|\mathcal{F}_t\vee \cF^F\right]\\&=\dE\left[\exp\left(-\alpha\int_{0+}^{t+s}F(u-)\rd X(u)-\alpha\int_{0+}^{t+s}(1-F(u-))\rd Y(u)\right)\big|\mathcal{F}_t\vee \cF^F\right]\re^{-k_\alpha(t+s)}
\notag\\
&=\dE\left[\exp\left(\int_{t}^{t+s} \varphi_X(i\alpha F(u))\rd u+\int_{t}^{t+s} \varphi_Y(i\alpha(1-F(u)))\rd u\right)\big|\mathcal{F}_t\vee \cF^F\right]\re^{-k_\alpha(t+s)}
\notag\\&\hspace{5cm}\times\exp\left(-\alpha\int_{0+}^{t}F(u-)\rd X(u)-\alpha\int_{0+}^{t}(1-F(u-))\rd Y(u)\right)
\notag\\
&=\exp\left(-\alpha\int_{0+}^{t}F(u-)\rd X(u)-\alpha\int_{0+}^{t}(1-F(u-))\rd Y(u)-k_{\alpha}(t)\right)
\notag\\
&=M_\alpha(t) \label{eq:martingale_equality} ,
\end{align}
where we made similar computations as in the proof of Lemma~\ref{lem:characteristic function V} but with conditional expectations. We deduce that $M_\alpha$ is an $\PAR{\mathcal{F}_t\vee \cF^F}_{t\geq 0}$-martingale. Taking expectations with respect to $\cF_t$, we obtain
\[
\dE\left[M_\alpha(t+s)\big|\mathcal{F}_t\right]=\dE\left[\dE\left[M_\alpha(t+s)\big|\mathcal{F}_t\vee \cF^F\right]\big|\mathcal{F}_t\right]=M_\alpha(t),
\]	
and thus $M_\alpha$ is also an $\PAR{\mathcal{F}_t}_{t\geq 0}$-martingale. 

Let $\tau$ be a $\PAR{\mathcal{F}_t\vee \cF^F}_{t\geq 0}$-stopping time bounded by a constant $k$.  By \eqref{eq:martingale_equality}, we note that $M_\alpha(t)=\dE\SBRA{M_\alpha(k)\,\vert\, \cF_t\vee \cF^F}$ for all $t\in[0,k]$. As $M_\alpha$ is a càdlàg $\PAR{\mathcal{F}_t\vee \cF^F}_{t\geq 0}$-martingale and $\tau$ is a bounded $\PAR{\mathcal{F}_t\vee \cF^F}_{t\geq 0}$-stopping time, by Doob's optional sampling theorem (see e.g.\ \cite[Theorem II.3.2]{revuz-yor99}), we have $M_\alpha(\tau)=\dE\SBRA{M_\alpha(k)\,\vert\, \cG_\tau}$, where
\[
\cG_\tau=\BRA{A\in\cG_\infty \colon \,  A\cap\BRA{\tau\leq t}\in\cG_t, \; \forall t\geq 0},
\]
in which $\cG_t = \mathcal{F}_t \vee \cF^F$ for all $t \geq 0$, and $\cG_\infty = \bigvee_{t\geq 0} \cG_t$. 
As $M_\alpha(0)=1$, taking conditional expectations with respect to $\cF^F$ in~\eqref{eq:martingale_equality} with $s=k$ and $t=0$, we easily observe that 
\[
\dE[M_\alpha(k)\,\vert\, \cF^F]=1 .
\]
As $\cF^F\subset \cG_\tau$, we deduce
\begin{align*}
\dE\SBRA{M_\alpha(\tau)\,\vert\, \cF^F}=\dE\SBRA{\dE\SBRA{M_\alpha(k)\,\vert\, \cG_\tau}\,\vert\, \cF^F}=\dE\SBRA{M_\alpha(k)\,\vert\, \cF^F}=1 .
\end{align*}

\small
\bibliographystyle{alpha}
\bibliography{biblio}

@article {AliliPatie10,
    AUTHOR = {Alili, Larbi and Patie, Pierre},
     TITLE = {Boundary-crossing identities for diffusions having the
              time-inversion property},
   JOURNAL = {J. Theoret. Probab.},
  FJOURNAL = {Journal of Theoretical Probability},
    VOLUME = {23},
      YEAR = {2010},
    NUMBER = {1},
     PAGES = {65--84},
      ISSN = {0894-9840,1572-9230},
   MRCLASS = {60J60 (60G18 60G40 60J65)},
       DOI = {10.1007/s10959-009-0245-3},
}

@article {AliliPatie14,
    AUTHOR = {Alili, Larbi and Patie, Pierre},
     TITLE = {Boundary crossing identities for {B}rownian motion and some
              nonlinear {ODE}'s},
   JOURNAL = {Proc. Amer. Math. Soc.},
  FJOURNAL = {Proceedings of the American Mathematical Society},
    VOLUME = {142},
      YEAR = {2014},
    NUMBER = {11},
     PAGES = {3811--3824},
      ISSN = {0002-9939,1088-6826},
   MRCLASS = {60J65 (34B24 35K05 60J50)},
  MRNUMBER = {3251722},
MRREVIEWER = {Zoran\ Vondra\v cek},
       DOI = {10.1090/S0002-9939-2014-12194-0},
       URL = {https://doi-org.proxy.bibliotheques.uqam.ca/10.1090/S0002-9939-2014-12194-0},
}

@book {Asmussen-Albrecher10,
    AUTHOR = {Asmussen, S{\o}ren and Albrecher, Hansj\"{o}rg},
     TITLE = {Ruin probabilities},
    SERIES = {Advanced Series on Statistical Science \& Applied Probability},
    VOLUME = {14},
   EDITION = {Second},
 PUBLISHER = {World Scientific Publishing Co. Pte. Ltd., Hackensack, NJ},
      YEAR = {2010},
     PAGES = {xviii+602},
      ISBN = {978-981-4282-52-9; 981-4282-52-9},
   MRCLASS = {60K15 (60H30 62P05 91B30)},
  MRNUMBER = {2766220},
MRREVIEWER = {Tina\ M.\ Marquardt},
       DOI = {10.1142/9789814282536},
       URL = {https://doi.org/10.1142/9789814282536},
}

@book{Pandemic-Insurance-book21,
author = {Boado Penas, Carmen and Eisenberg, Julia and Sahin, Sule},
year = {2021},
month = {11},
pages = {},
title = {Pandemics: Insurance and Social Protection},
isbn = {978-3-030-78336-5},
doi = {10.1007/978-3-030-78334-1}
}

@article{BotvichOConnell1995,
  author  = {Dmitri D. Botvich and Nicholas G. Duffield},
  title   = {Large deviations, the shape of the loss curve, and economies of large-scale multiplexers},
  journal = {Queueing Systems},
  volume  = {20},
  number  = {3},
  pages   = {293--320},
  year    = {1995},
  doi     = {10.1007/BF01149804}
}

@article {Chaumont-Pellas23,
    AUTHOR = {Chaumont, Lo\"ic and Pellas, Thomas},
     TITLE = {Creeping of {L}\'evy processes through curves},
   JOURNAL = {Electron. J. Probab.},
  FJOURNAL = {Electronic Journal of Probability},
    VOLUME = {28},
      YEAR = {2023},
     PAGES = {Paper No. 53, 25},
      ISSN = {1083-6489},
   MRCLASS = {60G51},
  MRNUMBER = {4574483},
       DOI = {10.1214/23-ejp942},
       URL = {https://doi-org.proxy.bibliotheques.uqam.ca/10.1214/23-ejp942},
}

@book{Dickson2017,
  author    = {David C. M. Dickson},
  title     = {Insurance Risk and Ruin},
  edition   = {2},
  publisher = {Cambridge University Press},
  address   = {Cambridge, UK},
  year      = {2017},
  isbn      = {9781107093480}
}

@article{DiekerMandjes2011,
  author    = {Ton Dieker and Michel Mandjes},
  title     = {Extremes of {M}arkov‐additive processes with one‐sided jumps, with queueing applications},
  journal   = {Methodology and Computing in Applied Probability},
  volume    = {13},
  number    = {2},
  pages     = {221--267},
  year      = {2011},
  doi       = {10.1007/s11009-009-9140-8}
}

@article {duffield-oconnell95,
    AUTHOR = {Duffield, Nicholas G. and O'Connell, Neil},
     TITLE = {Large deviations and overflow probabilities for the general
              single-server queue, with applications},
   JOURNAL = {Math. Proc. Cambridge Philos. Soc.},
  FJOURNAL = {Mathematical Proceedings of the Cambridge Philosophical
              Society},
    VOLUME = {118},
      YEAR = {1995},
    NUMBER = {2},
     PAGES = {363--374},
      ISSN = {0305-0041,1469-8064},
   MRCLASS = {60F10 (60K25 90B22)},
  MRNUMBER = {1341797},
MRREVIEWER = {J.\ Steinebach},
       DOI = {10.1017/S0305004100073709},
       URL = {https://doi-org.proxy.bibliotheques.uqam.ca/10.1017/S0305004100073709},
}

@article {Feng-Garrido11,
    AUTHOR = {Feng, Runhuan and Garrido, Jose},
     TITLE = {Actuarial applications of epidemiological models},
   JOURNAL = {N. Am. Actuar. J.},
  FJOURNAL = {North American Actuarial Journal},
    VOLUME = {15},
      YEAR = {2011},
    NUMBER = {1},
     PAGES = {112--136},
      ISSN = {1092-0277},
   MRCLASS = {91B30 (92D30)},
  MRNUMBER = {2792197},
       DOI = {10.1080/10920277.2011.10597612},
       URL = {https://doi.org/10.1080/10920277.2011.10597612},
}

@article {HerrmannTanre16,
    AUTHOR = {Herrmann, Samuel and Tanr\'e, Etienne},
     TITLE = {The first-passage time of the {B}rownian motion to a curved
              boundary: an algorithmic approach},
   JOURNAL = {SIAM J. Sci. Comput.},
  FJOURNAL = {SIAM Journal on Scientific Computing},
    VOLUME = {38},
      YEAR = {2016},
    NUMBER = {1},
     PAGES = {A196--A215},
      ISSN = {1064-8275,1095-7197},
       DOI = {10.1137/151006172},
}

@phdthesis{Ivanovs2011OnesidedMAP,
  author       = {Ivanovs, Jevgenijs},
  title        = {One‐sided Markov additive processes and related exit problems},
  school       = {Universiteit van Amsterdam},
  year         = {2011},
  address      = {Amsterdam, Netherlands},
  note         = {PhD thesis, Faculty of Science (FNWI), Korteweg-de Vries Institute for Mathematics (KdVI)},
  isbn         = {978-90-8891-311-2},
  url          = {https://hdl.handle.net/11245/1.360305}
}

@article{vanKreveldMandjesDorsman2024,
  author    = {Lucas van Kreveld and Michel Mandjes and Jan-Pieter Dorsman},
  title     = {Cram{\'e}r–{L}undberg asymptotics for spectrally positive {M}arkov additive processes},
  journal   = {Scandinavian Actuarial Journal},
  volume    = {2024},
  number    = {6},
  pages     = {561--582},
  year      = {2024},
  doi       = {10.1080/03461238.2023.2280287}
}

@article{vanKreveldMandjesDorsman2022,
  author       = {van Kreveld, Lucas and Mandjes, Michel and Dorsman, Jan-Pieter},
  title        = {Extreme Value Analysis for a {M}arkov Additive Process Driven by a Nonirreducible Background Chain},
  journal      = {Stochastic Systems},
  volume       = {12},
  number       = {3},
  pages        = {293--317},
  year         = {2022},
  doi          = {10.1287/stsy.2021.0086}
}

@book{Kyprianou2014,
  author    = {Andreas E. Kyprianou},
  title     = {Fluctuations of {L}{\'e}vy Processes with Applications: Introductory Lectures},
  edition   = {2},
  series    = {Universitext},
  year      = {2014},
  publisher = {Springer, Berlin},
  doi       = {10.1007/978-3-642-37632-0},
  url       = {https://link.springer.com/book/10.1007/978-3-642-37632-0}
}

@article {Ji-Zhang10,
    AUTHOR = {Ji, Lanpeng and Zhang, Chunsheng},
     TITLE = {The {G}erber-{S}hiu penalty functions for two classes of
              renewal risk processes},
   JOURNAL = {J. Comput. Appl. Math.},
  FJOURNAL = {Journal of Computational and Applied Mathematics},
    VOLUME = {233},
      YEAR = {2010},
    NUMBER = {10},
     PAGES = {2575--2589},
      ISSN = {0377-0427,1879-1778},
   MRCLASS = {91B30 (60J27)},
  MRNUMBER = {2577844},
       DOI = {10.1016/j.cam.2009.11.004},
       URL = {https://doi.org/10.1016/j.cam.2009.11.004},
}

@article {Lefevre-Simon-SIS22,
    AUTHOR = {Lef\`{e}vre, Claude and Simon, Matthieu},
     TITLE = {On the risk of ruin in a {SIS} type epidemic},
   JOURNAL = {Methodol. Comput. Appl. Probab.},
  FJOURNAL = {Methodology and Computing in Applied Probability},
    VOLUME = {24},
      YEAR = {2022},
    NUMBER = {2},
     PAGES = {939--961},
      ISSN = {1387-5841,1573-7713},
   MRCLASS = {91G05 (60J28 92D30)},
  MRNUMBER = {4446872},
       DOI = {10.1007/s11009-021-09924-z},
       URL = {https://doi.org/10.1007/s11009-021-09924-z},
}

@article {Lefevre-Simon-SIR21,
    AUTHOR = {Lef\`{e}vre, Claude and Simon, Matthieu},
     TITLE = {Ruin problems for epidemic insurance},
   JOURNAL = {Adv. in Appl. Probab.},
  FJOURNAL = {Advances in Applied Probability},
    VOLUME = {53},
      YEAR = {2021},
    NUMBER = {2},
     PAGES = {484--509},
      ISSN = {0001-8678,1475-6064},
   MRCLASS = {91G05 (60J28 92D30)},
  MRNUMBER = {4280455},
MRREVIEWER = {Xiaodong\ Bai},
       DOI = {10.1017/apr.2020.66},
       URL = {https://doi.org/10.1017/apr.2020.66},
}

@article {Li-Garrido05,
    AUTHOR = {Li, Shuanming and Garrido, Jos\'{e}},
     TITLE = {Ruin probabilities for two classes of risk processes},
   JOURNAL = {Astin Bull.},
  FJOURNAL = {Astin Bulletin. The Journal of the International Actuarial
              Association},
    VOLUME = {35},
      YEAR = {2005},
    NUMBER = {1},
     PAGES = {61--77},
      ISSN = {0515-0361},
   MRCLASS = {Expansion},
  MRNUMBER = {2142684},
       DOI = {10.2143/AST.35.1.583166},
       URL = {https://doi-org.proxy.bibliotheques.uqam.ca/10.2143/AST.35.1.583166},
}

@book {protter04,
    AUTHOR = {Protter, Philip E.},
     TITLE = {Stochastic integration and differential equations},
    SERIES = {Applications of Mathematics (New York)},
    VOLUME = {21},
   EDITION = {Second},
      NOTE = {Stochastic Modelling and Applied Probability},
 PUBLISHER = {Springer-Verlag, Berlin},
      YEAR = {2004},
     PAGES = {xiv+415},
      ISBN = {3-540-00313-4},
   MRCLASS = {60-02 (60G44 60H05 60H10 60H20)},
  MRNUMBER = {2020294},
MRREVIEWER = {Evelyn\ Buckwar},
}

@book{Schmidli2017,
  author    = {Hanspeter Schmidli},
  title     = {Risk Theory},
  series    = {Springer Actuarial Lecture Notes},
  year      = {2017},
  publisher = {Springer, Cham},
  doi       = {10.1007/978-3-319-72005-0},
  isbn      = {9783319720050}
}

@Book{revuz-yor99,
 Author = {Daniel {Revuz} and Marc {Yor}},
 Title = {{Continuous martingales and Brownian motion. 3rd ed}},
 FJournal = {{Grundlehren der Mathematischen Wissenschaften}},
 Journal = {{Grundlehren Math. Wiss.}},
 ISSN = {0072-7830; 2196-9701/e},
 Volume = {293},
 Edition = {3rd ed.},
 ISBN = {3-540-64325-7/hbk},
 Pages = {xiii + 602},
 Year = {1999},
 Publisher = {Berlin: Springer},
 Language = {English},
 MSC2010 = {60-02 60G44 60J65},
doi={10.1007/978-3-662-06400-9}
}

\end{document}